\documentclass[reqno,a4paper,11pt]{article}
\usepackage{color}
\usepackage[latin1]{inputenc}
\usepackage[english]{babel}
\usepackage{amsmath,amssymb,amsfonts,amsthm,enumerate}
\usepackage{mathrsfs}
\usepackage[T1]{fontenc}
\usepackage{epsfig}
\usepackage{graphicx}

\definecolor{refkeybis}{gray}{.65}
\definecolor{labelkeybis}{gray}{.65}
{\makeatletter
\def\SK@refcolor{\color{refkeybis}}%
\def\SK@labelcolor{\color{labelkeybis}}}

\numberwithin{equation}{section} 
\newtheorem{theorem}{Theorem}[section]
\newtheorem{lemma}[theorem]{Lemma}

\newtheorem{proposition}[theorem]{Proposition}

\newcommand{\N}{\mathbb{N}}
\newcommand{\R}{\mathbb{R}}
\newcommand{\C}{\mathbb{C}}

\newcommand{\Leb}[1]{{\mathscr L}^{#1}} 

\newcommand{\<}{\langle}
\renewcommand{\>}{\rangle}
\renewcommand{\a}{\alpha}

\newcommand{\e}{\varepsilon}

\newcommand{\p}{\partial}

\newcommand{\Probabilities}[1]{\mathscr P\bigl(#1\bigr)} 

\newcommand{\Error}[2]{{\mathscr E}_\varepsilon'(#1,#2)}
\newcommand{\Errordelta}[2]{{\mathscr E}_\varepsilon(#1,#2)}

 \newcommand{\bb}{{\mbox{\boldmath$b$}}}

 \newcommand{\tauV}{{\kern-3pt\tau}}

 \newcommand{\oVVVk}{\overline{\mbox{\boldmath$V$}}\kern-3pt}
 \newcommand{\tVVVk}{\tilde{\mbox{\boldmath$V$}}\kern-3pt}

 \newcommand{\eps}{\varepsilon}

\newcommand{\esssup}{\mathop{\rm ess \, sup}}

 \newcommand{\LL}{\mathscr{L}}

 \newcommand{\WW}{\mathscr{W}}
\def\signaf{\bigskip \begin{center} {\sc Alessio Figalli\par\vspace{3mm}
Department of Mathematics\\
The University of Texas at Austin\\
1 University Station, C1200\\
Austin TX 78712, USA\\
email:} \texttt{figalli@math.utexas.edu} \end{center}}

\def\signml{\bigskip\begin{center} {\sc Marilena Ligab\`o\par\vspace{3mm}
Dipartimento di Matematica\\
Universit\`a degli studi di Bari\\
Via E. Orabona, 4\\
70125 Bari, Italy\\
email:} \texttt{ligabo@dm.uniba.it}
\end{center}}

\def\signtp{\bigskip \begin{center} {\sc Thierry Paul\par\vspace{3mm}
Centre de math\'ematiques Laurent Schwartz - UMR 7640\\
Ecole Polytechnique\\
Palaiseau 91128, France\\
email:} \texttt{thierry.paul@math.polytechnique.fr}
\end{center}}

\begin{document}


 \title{Semiclassical limit for mixed states\\ with singular and rough potentials}
 \author{Alessio Figalli, Marilena Ligab\`o and Thierry Paul}

\footnotetext{AF is partially supported by the NSF grant DMS-0969962.}

 \maketitle
\date{}
\begin{abstract}
We consider the semiclassical limit for the Heisenberg-von Neumann equation with a potential which consists of the sum of a repulsive Coulomb potential, plus a Lipschitz potential
whose gradient belongs to $BV$; this assumption on the potential guarantees the well posedness of the Liouville equation in the space of bounded integrable solutions.
We find sufficient conditions on the initial data to ensure that the quantum dynamics converges to the classical one.
More precisely, we consider the Husimi functions of the solution of the Heisenberg-von Neumann equation, and under suitable assumptions on the initial data
we prove that they converge, as $\e \to 0$,
to the unique bounded solution of the Liouville equation (locally uniformly in time).
\end{abstract}


\section{Introduction}\label{intro}
The aim of this paper is to study the semiclassical limit for the Heisenberg-von Neumann (quantum Liouville) equation:
\begin{equation} \label{eq:Hei} \left\{
\begin{array}{l} i\e \p_t \tilde\rho_t^\e=[H_\e,\tilde\rho_\e^t],\ \ \ \ \ \ \ \ \ \ \ \ \ \ \ \ \ \ \ \ \ \\\\
\tilde\rho^\e_0=\tilde\rho_{0,\e},
\end{array}
\right.
\end{equation}
$\{\tilde\rho_{0,\e}\}_{\e>0}$ being a family of uniformly bounded (with respect to $\e$), positive, trace class operators, and with $H_\e=-\frac{\e^2}{2}\Delta+U$.

When $\tilde\rho_{0,\e}$ is the orthogonal projector onto $\psi_{0,\e}\in L^2(\R^n)$, (\ref{eq:Sch}) is equivalent (up to a global phase)  to the
Schr\"odinger equation
\begin{equation} \label{eq:Sch} \left\{
\begin{array}{l} i\e \p_t \psi^\e_t=-\frac{\e^2}{2}\Delta \psi^\e_t+U\psi^\e_t
=H_\e\psi^\e_t,\\\\
\psi^\e_0=\psi_{0,\e}\in L^2(\R^n),
\end{array}
\right.
\end{equation}
We recall that the Wigner transform $W_\e\psi$ of a function $\psi
\in L^2(\R^n)$ is defined as
$$
W_\e\psi(x,p):=\frac{1}{(2\pi)^n}\int_{\R^n}\psi(x+\frac{\e}{2}y)
\overline{\psi(x-\frac{\e}{2}y)}e^{-ipy}dy,
$$
and the one of a density matrix $\tilde\rho$ is defined as
\begin{equation}\label{Wignerdensity}
W_\e\rho(x,p):=\frac{1}{(2\pi)^n}\int_{\R^n}\rho(x+\frac{\e}{2}y,
x-\frac{\e}{2}y)e^{-ipy}dy,
\end{equation}
where $\rho(x,x')$ denotes the integral kernel associated to the operator $\tilde\rho$.

The weak limit of the Wigner function of the solution of (\ref{eq:Sch}) or (\ref{eq:Hei}) has
been studied in many articles (e.g. \cite{lionspaul,gerard,gerard3}, and more recently
in  strong topology in \cite{athanassoulis3,athanassoulis2}).
More precisely, it is well-known that the limit dynamics of the Schr\"odinger equation is related to the Liouville equation
\begin{equation}
\label{eq:liouville}
\p_t \mu + p \cdot \nabla_x\mu- \nabla U(x)\cdot \nabla_p \mu=0,
\end{equation}
and, roughly speaking, the above results state that:\\
(A) If $U$ is of class $C^2$ and there exists a sequence $\e_k \to 0$ such that $W_{\e_k}\rho_{0,\e_k}$ converges in the sense of distribution
to some (nonnegative) measure $\mu_0$, then $W_{\e_k}\rho_t^{\e_k} \to (\Phi_t)_\# \mu_0$ (the convergence is again in the sense of distribution), where $\Phi_t$ is the
(unique) flow map
associated to the Hamiltonian system
\begin{equation}
\label{eq:ODE}
\left\{ \begin{array}{l}
        \dot x = p,\\
	\dot p=-\nabla U(x)
        \end{array}
\right.
\end{equation}
so that $\mu_t:=(\Phi_t)_\# \mu_0$ is the unique solution to \eqref{eq:liouville} (here and in the sequel, $\#$ denotes the push-forward, so that $\mu_t(A)=\mu_0(\Phi_t^{-1}(A))$ for all $A\subset \R^{2n}$ Borel).\\
(B) If $U$ is of class $C^1$ and there exists a sequence $\e_k \to 0$ such that the curve $t \mapsto W_{\e_k}\rho_t^{\e_k}$ converges in the sense of distribution
to some curve of (nonnegative) measure $t \mapsto \mu_t$, then $\mu_t$ solves \eqref{eq:liouville}.

In the present paper we want to use some recent results proved in \cite{amfiga2,AFFGP} to improve the
literature in two directions:
\begin{enumerate}
\item By lowering the regularity assumptions of (A) on the potential in order get convergence results for a more general class of potentials, as described below.

\item Get rid of the ``after an extraction of a subsequence'' argument, due to compactness,
used in most of the available proofs
where one is unable to uniquely identify the limit.
More precisely, in (B) above one needs to take a subsequence along which the whole curve $t \mapsto W_{\e_k}\rho_t^{\e_k}$
converges  for all $t$ in order to obtain a solution to \eqref{eq:liouville}. Moreover, the limiting solution may depend on the particular subsequence.
In our case we will be able to show that, for a class potential much larger than $C^2$,
once one assumes that the Wigner functions at time $t=0$ have a limit,
then the limit at any other time will converge to a ``uniquely identified'' solution of \eqref{eq:liouville}.
\end{enumerate}
The price to pay for the lack of regularity of the potential will be to have some size condition on the initial datum which
forbids the possibility of
 considering pure states. Even more, the Wigner function of the initial datum cannot concentrate at a point, a possibility
 which might actually enter in conflict with the fact that the underlying flow is not uniquely defined everywhere. Let us mention
 however that, with extra assumptions on the potential (but still allowing the possibility of not
 having uniqueness of a classical flow), it is possible to consider concentrating initial Wigner functions, giving rise to
 atomic measures whose evolution follows the ``multicharacteristics'' of the flow (see \cite{athanassoulis2}).

As described below, we will nevertheless show that, for general bounded and globally Lipschitz potential
 associated to locally $BV$ vector fields (in addition to some Coulomb part), the Wigner measure of the solution at any time
 is the push-forward of the initial one by the Ambrosio-DiPerna-Lions flow \cite{lions,ambrosio}.


Our method will use extensively the Husimi transforms
$\psi\mapsto\tilde{W}_\e\psi$ and $\rho\mapsto\tilde{W}_\e\rho$, which we recall are defined
in terms of convolution of the Wigner transform with the
$2n$-dimensional Gaussian kernel with variance $\e/2$:
\begin{equation}\label{heatK}
\tilde W_\e\psi:=(W_\e\psi)\ast G^{(2n)}_\e,\quad
\tilde W_\e\rho:=(W_\e\rho)\ast G^{(2n)}_\e,\quad G_\e^{(2n)}(x,p):=
\frac{e^{-(|x|^2+|p|^2)/\e}}{(\pi\e)^n}=G^{(n)}_\e(x)G^{(n)}_\e(p).
\end{equation}
Of course, the asymptotic behaviour of the Wigner and Husimi transform is the same in the limit $\e\to 0$. However, one of the
main advantages of the Husimi transform is that it is
nonnegative (see Appendix).

Let us observe that, thanks to \eqref{eq:husimi}, the $L^\infty$-norm of $\tilde W_\e\psi$ can be
estimated using the Cauchy-Schwarz inequality:
$$
\tilde W_\e\psi(x,p)\leq \frac{1}{\e^n}\|\psi\|^2_{L^2}
\|\phi^\e_{x,p}\|_{L^2}^2=\frac{\|\psi\|^2_{L^2}}{\e^n}.
$$
However, this estimate blows up as $\e\to 0$.
On the other hand we will prove that,
by averaging the initial condition with respect to translations,
we can get a uniform estimate as $\e\to 0$ (Section \ref{sect:toplitz}). This gives us, for instance, an important family of initial data to which our result
and the ones in \cite{AFFGP} apply (see also the other examples in Section
\ref{sect:examples}).

\section{The main results}
\subsection{Setting}
\label{sect:setting}
We are concerned with the derivation of classical
mechanics from quantum mechanics, corresponding to the study of the
asymptotic behaviour of solutions $\tilde{\rho}^\e_t$ to the
Heisenberg-von Neumann equation
\begin{equation}
\label{eq:Sch1} \left\{
\begin{array}{l} i\e \p_t \tilde{\rho}^\e_t=[H_\e,\tilde{\rho}^\e_t]\\\\
\tilde{\rho}^\e_0=\tilde{\rho}_{0,\e},
\end{array}
\right.
\end{equation}
as $\e\to 0$, where $H_\e=-\frac{\e^2}{2}\Delta+U$, and $U:\R^n \to \R$ is of the form $U_b+U_s$ on $\R^n$, where $U_s$ is a repulsive Coulomb potential
$$
U_s(x)=\sum_{1\leq i < j \leq M} \frac{Z_i Z_j}{|x_i-x_j|}, \qquad M \leq n/3,\,  x=(x_1, \dots x_M, \bar x) \in (\R^3)^M \times \R^{n-3M}, \, Z_i >0,
$$
$U_b$ is globally  bounded, locally Lipschitz, $\nabla U_b\in BV_{\textrm{loc}}(\R^n;\R^n)$, and
$$
\esssup_{x \in \R^{n}}\, \frac{|\nabla U_b(x)|}{1+|x|} <+\infty.
$$
The formal solution of (\ref{eq:Sch}) is $\tilde{\rho}^{\e}_t$, where
$$
\tilde{\rho}^{\e}_t:=e^{-itH_{\e}/\e}\tilde{\rho}_{0,\e}e^{-itH_{\e}/\e}
$$
and its kernel is $\rho^{\e}_t$. Moreover, as shown for instance in \cite{lionspaul}, $W_\e{\rho_t^\e}$ solves in the sense of distributions the
equation
\begin{equation}\label{Wigner_calc_in}
\partial_t W_\e\rho^\e_t+p\cdot\nabla_x W_\e\rho^\e_t
=\Errordelta{U}{\rho^\eps_t},
\end{equation}
where $\Errordelta{U}{\rho}$ is given by
\begin{equation}\label{error_delta}
\Errordelta{U}{\rho}(x,p):= -\frac{i}{(2\pi)^n}\int_{\R^n} \biggl[
\frac{U(x+\tfrac{\e}{2}y)-U(x-\tfrac{\e}{2}y)}{\e}\biggr]
\rho(x+\frac{\e}{2}y,x-\frac{\e}{2}y)e^{-ipy}dy.
\end{equation}
Adding and subtracting $\nabla U(x)\cdot y$ in the term in square
brackets and using $ye^{-ip\cdot y}=i\nabla_p e^{-ip\cdot y}$, an
integration by parts gives $\Errordelta{U}{\rho}=\nabla U(x)\cdot
\nabla_pW_\e\rho+\Error{U}{\rho}$, where $\Error{U}{\rho}$ is
given by
\begin{equation}\label{error}
\Error{U}{\rho}(x,p):= -\frac{i}{(2\pi)^n}\int_{\R^n} \biggl[
\frac{U(x+\tfrac{\e}{2}y)-U(x-\tfrac{\e}{2}y)}{\e}-\nabla
U(x)\cdot y \biggr] \rho(x+\frac{\e}{2}y,x-\frac{\e}{2}y)e^{-ipy}dy.
\end{equation}
Let $\bb:\R^{2n}\to\R^{2n}$ be the autonomous divergence-free vector field $\bb(x,p):=\bigl(p,-\nabla U(x)\bigr)$.
Then, by the discussion above, $W_\e\rho^\e_t$ solves the Liouville equation associated to $\bb$ with an error
term:
\begin{equation}\label{Wigner_calc}
\partial_t W_\e\rho^\e_t+\bb \cdot \nabla
W_\e\rho^\e_t =\Error{U}{\rho^\eps_t}.
\end{equation}
On the other hand, thanks to \eqref{Wigner_calc_in}, it is not difficult to prove that $\tilde{W}_{\e}\rho^{\e}_t$ solves in the sense of distributions the
equation
\begin{equation}\label{pdehusimi}
\partial_t\tilde W_\e\rho^\e_t+p\cdot\nabla_x\tilde W_\e\rho^\e_t=
\Errordelta{U}{\rho^\e_t}\ast G^{(2n)}_\e-\sqrt{\e}\nabla_x\cdot
[W_\e\rho^\e_t\ast \bar{G}_\e^{(2n)}],
\end{equation}
where
\begin{equation}\label{pdehusimibis}
\bar{G}_\e^{(2n)}(y,q):=\frac{q}{\sqrt{\e}}G^{(2n)}_\e(y,q).
\end{equation}
Since $W_\e\rho_t^\e$ and $\tilde W_\e\rho_t^\e$ have the same limit points as $\e \to 0$, the
heuristic idea is that in the limit $\e\to 0$ all error terms should disappear, and we should be left with the Liouville equation (which describes the classical dynamics)
$$
\partial_t \omega_t+\bb \cdot \nabla \omega_t=0\qquad \text{on }\R^{2n}.
$$

\subsection{Preliminary results on the Liouville equations}
Under the above assumptions on $U$ one cannot hope for a general uniqueness result in the space of measures for the Liouville equation,
as this would be equivalent to uniqueness for the ODE with vector field $\bb$ (see for instance \cite{cetraro}).
On the other hand, as shown in \cite[Theorem 6.1]{AFFGP}, the equation
\begin{equation}\label{eqn:continuity}
\left\{
\begin{array}{l} \partial_t \omega_t+\bb \cdot \nabla \omega_t=0 \ \ \ \ \ \ \ \ \ \ \ \ \ \ \ \ \ \ \ \ \ \\\\
\omega_0=\bar{\omega} \in L^{1}(\R^{2n})\cap L^{\infty}(\R^{2n})  \textrm{ and nonnegative},
\end{array}
\right.
\end{equation}
has existence and uniqueness in the space
$L^{\infty}_+([0,T];L^{1}(\R^{2n})\cap L^{\infty}(\R^{2n}))$.
This means that there exist a unique $\WW :[0,T] \to L^{1}(\R^{2n})\cap L^{\infty}(\R^{2n})$,
nonnegative and such that $\esssup_{t \in[0,T]}\|\WW_t\|_{L^1(\R^{2n})}+\|\WW_t\|_{L^\infty(\R^{2n})}  < +\infty$, that solves (\ref{eqn:continuity}) in the sense of distributions on $[0,T]\times \R^{2n}$.

One may wonder whether, in this general setting, solutions to the transport equation can still be described using the theory of characteristics. Even if in this case one cannot solve uniquely the ODE, one can still prove that there exists a unique flow map in the ``Ambrosio-DiPerna-Lions sense''. Let us recall the definition of \textit{Regular Lagrangian Flow}
(in short RLF) in the sense of Ambrosio-DiPerna-Lions:

We say that a (continuous) family of maps $\Phi_t:\R^{2n}\to \R^{2n}$, $t \geq 0$, is a RLF associated to \eqref{eq:ODE} if:
\begin{enumerate}
\item[-] $\Phi_0$ is the identity map.
\item[-] For $\LL^{2n}$-a.e. $(x,p)$, $t \mapsto \Phi_t(x,p)$ is an absolutely continuous curve solving \eqref{eq:ODE}.
\item[-] For every $T>0$ there exists a constant $C_T$ such that $(\Phi_t)_\# \LL^{2n}\leq C_T \LL^{2n}$ for all $t \in [0,T]$,
\end{enumerate}
where $\LL^{2n}$ denotes the Lebesgue measure on $\R^{2n}$. \\
Observe that, since $\nabla U$ is not Lipschitz, a priori the ODE \eqref{eq:ODE} could have more than one solution for some initial condition.
However, the approach via RLFs allows to get rid of this problem by looking at solutions to \eqref{eq:ODE} as a whole, and under suitable assumptions
on $U$ the RLF associated to \eqref{eq:ODE} exists, and it is unique in the following sense:
assume that $\Phi^1$ and $\Phi^2$ are two RLFs. Then, for $\LL^{2n}$-a.e. $(x,p)$, $\Phi_t^1(x,p)=\Phi_t^2(x,p)$ for all $t \in [0,+\infty).$ In particular, as shown in \cite[Section 6]{AFFGP},
the unique solution to \eqref{eqn:continuity} is given by
\begin{equation}
\label{eq:sol push forward}
\omega_t \LL^{2n}=(\Phi_t)_\# \bigl(\bar\omega \LL^{2n} \bigr).
\end{equation}

Hence, the idea is that, if we can ensure that any limit point of the Husimi transforms $\tilde W_\e\rho^\e_t$ give rives to a curve of measure belonging to
$L^{\infty}_+([0,T];L^{1}(\R^{2n})\cap L^{\infty}(\R^{2n}))$, by the aforementioned result we would deduce that the limit is unique (once the limit initial datum is fixed), and moreover it is transported by the unique RLF.
In order to get such a result we need to make some assumptions on the initial data.

\subsection{Assumptions on the initial data and main theorem}
\label{sect:assumptions and theorem}
Let $\{\tilde{\rho}_{0,\e}\}_{\e \in (0,1)}$ be a family of initial data which satisfy
$$
\tilde{\rho}_{0,\e}=\tilde{\rho}_{0,\e}^{*}, \quad \tilde{\rho}_{0,\e} \geq 0 \quad \textrm{and} \quad \textrm{tr}(\tilde{\rho}_{0,\e})=1\qquad \forall\,\e\in(0,1).
$$
Let
$$
\tilde{\rho}_{0,\e}=\sum_{j\in \N} \mu_{j}^{(\e)} \langle \phi_{j}^{(\e)}, \cdot \rangle  \phi_{j}^{(\e)}
$$
be the spectral decomposition of $\tilde{\rho}_{0,\e}$, and denote by $\rho_{0,\e}$ its integral kernel.

We assume:
\begin{equation}\label{eqn:hypH2}
\sup_{\e\in (0,1)} \sum_{j\in \N} \mu_{j}^{(\e)} \|H_{\e}\phi_{j}^{(\e)} \|^2 < +\infty,
\end{equation}
\begin{equation}\label{eqn:disop}
\frac{1}{\e^n} \tilde{\rho}_{0,\e} \leq  C \,\textrm{Id},
\end{equation}
\begin{equation}\label{eqn:tightness}
\lim_{R \to +\infty} \sup_{\e\in (0,1)} \int_{\R \setminus B_{R}^{(n)}} \rho_{0,\e}(x,x) \; dx =0,
\end{equation}
\centerline{and}
\begin{equation}\label{eqn:tightness2}
\lim_{R \to +\infty} \sup_{\e\in(0,1)} \frac{1}{(2\pi\e)^n} \int_{\R \setminus B_{R}^{(n)}} \mathcal F\rho_{0,\e}\Bigl(\frac{p}{\e},\frac{p}{\e}\Bigr) \; dp =0,
\end{equation}
where $B_{R}^{(n)}$ is the ball of radius $R$ in $\R^n$ and $\mathcal F$ is the Fourier transform on $\R^{2n}$, see \eqref{eq:Fourier}.
Conditions \eqref{eqn:tightness} and \eqref{eqn:tightness2} are equivalent to asking
that the family of probability measure
$\{\tilde W_\e\rho_{0,\e}\}_{\e\in(0,1)}$ is tight (see Appendix). By Prokhorov's Theorem, this is equivalent to the compactness of $\{\tilde W_\e\rho_{0,\e}\}_{\e\in(0,1)}$ with respect to the weak topology of probability measures (i.e., in the duality with $C_b(\R^{2n})$, the space of bounded continuous functions).
Hence, up to extracting a subsequence, assumptions \eqref{eqn:tightness} together with \eqref{eqn:tightness2} is equivalent to the existence of a probability density $\bar\omega$ such that
\begin{equation}\label{eqn:exlimmeas}
w-\lim_{\e \to 0} \tilde W_\e\rho_{0,\e} \Leb{2n} =\bar{\omega} \Leb{2n}\in \Probabilities{\R^{2n}},
\end{equation}
where $\Probabilities{\R^{2n}}$ denotes the space of probability measure on $\R^n$.
In order to avoid a tedious notation which would result by working with a subsequence $\e_k$,
we will assume that \eqref{eqn:exlimmeas} holds along the whole sequence $\e\to 0$,
keeping in mind that all the arguments could be repeated with an arbitrary subsequence.

Let us observe that condition \eqref{eqn:hypH2} is slightly weaker than $\sup_{\e \in (0,1)}\textrm{tr}(H_{\e}^2\tilde{\rho}_{0,\e}) <+\infty,$ as in order to give a sense to the latter
we need the operator $H_{\e}^2\tilde{\rho}^t_{\e}$ to make sense (at least on a core).
Concerning assumption \eqref{eqn:exlimmeas},  let us observe that
the hypothesis $\textrm{tr}(\tilde{\rho}_{0,\e})=1$ implies that $\tilde W_\e\rho_{0,\e} \in \Probabilities{\R^{2n}}$ (see Appendix).

To express in a better and cleaner way the fact that the convergence is uniform in time, we denote by $d_{{\mathscr P}}$ any bounded distance inducing the weak topology in $\Probabilities{\R^{2n}}$.
Recall also that $\Phi_t$ denotes the unique RLF associated to
$\bb(x,p)=(p,-\nabla U(x))$, so that $(\Phi_t)_\# \bigl(\bar\omega
\LL^{2n} \bigr)$ is the unique nonnegative solution of
\eqref{eqn:continuity} in $L^{\infty}_+([0,T];L^{1}(\R^{2n})\cap
L^{\infty}(\R^{2n}))$.
\begin{theorem}\label{thm:mainthm}
Let $U$ be as in Section \ref{sect:setting}.
Under the assumptions (\ref{eqn:hypH2}), (\ref{eqn:disop}) and (\ref{eqn:exlimmeas})
\begin{equation}\label{eqn:thesis}
\lim_{\e \to 0} \sup_{[0,T]} d_{{\mathscr P}}(\tilde{W}_{\e}\rho_{t}^{\e}\Leb{2n},(\Phi_t)_\# \bigl(\bar\omega \LL^{2n} \bigr))=0.
\end{equation}
Moreover, if we define $\WW_t \Leb{2n}=(\Phi_t)_\# \bigl(\bar\omega \LL^{2n} \bigr)$, for every smooth function $\varphi \in C^\infty_c(\R^{2n})$ the map $t \mapsto \int_{\R^{2n}}\varphi \WW_t \,dx\,dp$
is continuously differentiable, and
$$
\frac{d}{dt} \int_{\R^{2n}}\varphi \WW_t \,dx\,dp=\int_{\R^{2n}} \bb \cdot \nabla \varphi  \WW_t\, dx\,dp.
$$
\end{theorem}
The rest of the paper will be concerned with the proof of Theorem \ref{thm:mainthm}. However, before proceeding with the proof, we first provide some example and sufficient conditions
for our result to apply.

\section{Examples}
\label{sect:examples}
We will give three types of examples of density matrices satisfying the assumptions of the preceding section, so that Theorem \ref{thm:mainthm} applies.

\subsection{Average of an orthonormal basis}
For simplicity, we set up our first example in the one-dimensional case. In particular,
there is no Coulomb interaction (that is, $U=U_b$), since by assumption Coulomb interactions
are three-dimensional.
We leave to the interested
reader the extension to arbitrary dimension (the only
difference in the case $U_s \neq 0$ appears when checking assumption \eqref{eqn:hypH2}).

Let us consider the orthonormal basis of $L^2(\R)$ given by the (semiclassical) Hermite functions
$$
\psi_j^{(\e)}(x)=\frac{e^{-x^2/2\e}}{\sqrt{2^j j!}(\pi \e)^{1/4}}H_j\left(\frac{x}{\sqrt{\e}}\right), \qquad j \in \N,
$$
where $H_j$'s are the Hermite polynomials, i.e.
$$
H_j(x)=(-1)^je^{x^2}\frac{d^j}{dx^j}e^{-x^2}.
$$
The following holds:

\begin{proposition}\label{herm}
Let $\{\mu_j^{(\e)}\}_{j\in \N}$ be a sequence of positive numbers, and define the
density matrix $\rho_\e$ given by
\[
\rho_\e=\sum_{j \in \N} \mu_j^{(\e)}\langle \psi_j^{(\e)}, \cdot \rangle \psi_j^{(\e)}.
\]
Assume that
\begin{itemize}
\item $0\leq \mu_j^{(\e)}\leq C\e$, $\sum_{j \in \N} \mu_j^{(\e)}=1$;
\item $\e^{2}\sum_{j \in \N} \mu_j^{(\e)} j^{2}\leq C<+\infty$;
\item $w-\lim_{\e \to 0} \sum_{j \in \N}\mu_j^{(\e)}\delta(x^2+p^2-j\e)=\bar{\omega}\LL^{2} \in \Probabilities{\R^2}$.
\end{itemize}
Then \eqref{eqn:hypH2}, \eqref{eqn:disop}, and \eqref{eqn:exlimmeas} hold.
\end{proposition}
\begin{proof}
The first assumption is equivalent to \eqref{eqn:disop} and the trace-one condition.

Concerning \eqref{eqn:hypH2}, using
the well-know fact that
$$
\e\frac{d}{dx}\psi_j^{(\e)}=\sqrt{\frac{\e}{2}}\left(\sqrt{j}\psi_{j-1}^{(\e)}-\sqrt{j+1}\psi_{j+1}^{(\e)} \right),
$$
by a simple calculation it follows that
\begin{eqnarray*}
H_{\e}\psi_j^{(\e)}& = &-\frac{\e^2}{2}\frac{d^2}{dx^2}\psi_j^{(\e)}+U_b\psi_j^{(\e)}\\
                      & = &-\frac{\e}{4}\left(\sqrt{j(j-1)}\psi_{j-2}^{(\e)}-(2j+1)\psi_{j}^{(\e)}+\sqrt{(j+1)(j+2)}\psi_{j+2}^{(\e)}\right)+U_b\psi_j^{(\e)}.
\end{eqnarray*}
Hence
\begin{eqnarray*}
\|H_{\e}\psi_j^{(\e)}\|^2&\leq &\left[\frac{\e}{4}\left(\sqrt{j(j-1)}+ (2j+1)+\sqrt{(j+1)(j+2)}\right)+\|U_b\|_{\infty}\right]^2\\
&\leq & C(1+\e^2 j^2),
\end{eqnarray*}
and  \eqref{eqn:hypH2} follows
from the first two assumptions.

Finally, the third assumption implies \eqref{eqn:exlimmeas} by noticing that
$$w-\lim_{\e \to 0, j\to \infty, j\e\to a} \tilde W_\e\psi_j^{(\e)}=\delta(x^2+p^2-a)\qquad \forall \,a\geq 0
$$
(see, for instance,  \cite[Exemple III.6]{lionspaul}).
\end{proof}

\subsection{T\"oplitz case}
\label{sect:toplitz}
Let $\phi\in H^2(\R^n;\C)$ with $\int_{\R^n}|\phi(x)|^2\,dx=1$. Given $\epsilon,\varsigma>0$, for any
$w,q\in\R^n$ let $\psi^\e_{w,q}$ be defined by
$$
\psi^\e_{w,q}(x):=\frac{1}{\varsigma^{n/2}}\phi\biggl(\frac{x-q}{\varsigma}
\biggr)e^{i\frac{w\cdot x}\e}.
$$
Then, using Plancherel theorem, one can easily check that the identity
\begin{equation}\label{eq:Linfty bound}
\frac{1}{\e^n}\int_{\R^{2n}}|\psi^\e_{w,q}\rangle\langle \psi^\e_{w,q}|\,dw\,dq=
 (2\pi)^n{\rm Id}
\end{equation}
holds, where $|\psi\rangle\langle \psi|$ is the Dirac notation for the orthogonal projection onto a normalized vector $\psi \in L^2(\R^n)$.
Thanks to \eqref{eq:Linfty bound} and the fact that orthogonal projectors are nonnegative operators, we immediately obtain the
following important estimate: for every nonnegative bounded function
$\chi_\e:\R^{2n}\to \R$, it holds
\begin{equation}\label{linfini}
\frac{1}{\e^n}\int_{\R^{2n}}\chi_\e(w,q)|\psi^\e_{w,q}\rangle\langle \psi^\e_{w,q}|\,dw\,dq\leq
\|\chi_\e\|_{\infty}(2\pi)^n{\rm Id}.
\end{equation}
Set now
$$
\tilde{\rho}_{0,\e}:=\int_{\R^{2n}}\chi_\e(w,q)|\psi^\e_{w,q}\rangle\langle \psi^\e_{w,q}|\,dw\,dq, \qquad \e \in (0,1),
$$
where $\{\chi_\e\}_{\e\in (0,1)}$ is a family of nonnegative bounded functions such that $\int_{\R^{2n}}\chi_\e(w,q)\,dw\,dq=1$,
and let $S$ be the singular set of $U_s$ as defined in \eqref{defn:S} below.
\begin{proposition}\label{top}
 Let $\varsigma=\varsigma(\e)=\e^\a$ with $\a \in (0,1)$, and assume that
\begin{itemize}
\item $\sup_{\e\in (0,1)} \|\chi_\e\|_{\infty}<+\infty$.
\item $w-\lim_{\e \to 0} \chi_\e \Leb{2n} =\bar{\omega} \Leb{2n}\in \Probabilities{\R^{2n}}$
\item
$\int_{\R^{2n}}\chi_\e(w,q)\left( |w|^4+ \frac{1}{{\rm dist}(q,S)^2}\right)\,dw\,dq\leq C<+\infty$.
\end{itemize}
Then \eqref{eqn:hypH2}, \eqref{eqn:disop}, and \eqref{eqn:exlimmeas} hold for the family of initial data $\{\tilde{\rho}_{0,\e}\}_{\e \in (0,1)}$.
\end{proposition}
\begin{proof}
\eqref{eqn:disop} follows from the first assumption and \eqref{linfini}.

Since $\varsigma=\e^\a$ with $\a \in (0,1)$ we have that for all $(w,q)\in \R^{2n}$
$$
w-\lim_{\e \to 0} \tilde W_\e \psi^\e_{w,q}\Leb{2n}=\delta_{(w,q)},
$$
see \cite[Exemple III.3]{lionspaul},
and so \eqref{eqn:exlimmeas} follows from our second assumption.

To show that the third assumption implies \eqref{eqn:hypH2}, we notice that in this case \eqref{eqn:hypH2} can be written as follows
\begin{equation}\label{eqn:weaker}
\int_{\R^{2n}}\chi_\e(w,q)\langle H_\e \psi^\e_{w,q},H_{\e}\psi^\e_{w,q}\rangle\,dw\,dq < +\infty.
\end{equation}
Since $\a < 1$, and $\phi\in H^2(\R^n;\C)$,
by a simple computation we get
\begin{align*}
\langle H_{\e}\psi^\e_{w,q},H_{\e}\psi^\e_{w,q}\rangle
&\leq \frac{\e^4}{2} \langle \Delta_x \psi^\e_{w,q},\Delta_x\psi^\e_{w,q}\rangle + 2 \langle U \psi^\e_{w,q},U \psi^\e_{w,q}\rangle\\
&\leq C\bigl(1+|w|^4\bigr)+ C \int_{\R^n} U(x)^2 \frac{1}{\varsigma^{n}}\phi^2\biggl(\frac{x-q}{\varsigma}
\biggr)\,dx.
\end{align*}
Since $U_b$ is bounded, $|U_s(q)| \leq C/{\rm dist}(q,S)$, and $\int_{\R^n}|\phi(x)|^2\,dx=1$,
a simple estimate analogous to the one in Section \ref{sect:sing} shows that \eqref{eqn:hypH2} holds. We leave the details to the interested reader.
\end{proof}.

\subsection{Conditions on the Wigner function}
Here we consider a general family of density matrices $\{\tilde{\rho}_{0,\e}\}_{\e \in(0,1)}$ which satisfies the tightness conditions
\eqref{eqn:tightness} and \eqref{eqn:tightness2} (so that \eqref{eqn:exlimmeas} is satisfied up to
the extraction of a subsequence). In the next proposition we show some simple sufficient conditions on the Wigner functions $\{W_{\e}\rho_{0,\e}\}_{\e \in(0,1)}$ in order to ensure the validity of assumptions (\ref{eqn:hypH2}) and (\ref{eqn:disop}).
\begin{proposition}\label{cv}
Assume that
\begin{itemize}
\item $\max_{\vert\alpha\vert,\vert\beta\vert\leq [\frac n2]+1}\
 \|\partial_x^\alpha\partial_p^\beta W_\e\rho_{0,\e}\|_{\infty}\leq C<+\infty$,
\item $\int_{\R^{2n}}\left(\frac{|p|^4}{4}+U^2(x)+|p|^2U(x)-\frac{n\e^2}{2}\Delta U(x)\right)W_\e\rho_{0,\e}(x,p)\,dx\,dp\leq C<+\infty$.
\end{itemize}
Then \eqref{eqn:hypH2} and  \eqref{eqn:disop} hold.
\end{proposition}
\begin{proof}
Let us recall first that the Weyl symbol of an operator $\tilde{\rho}$ of integral kernel $\rho(x,y)$ is, by definition, given by
$$
\sigma_\e(\tilde{\rho})(x,p):=\int_{\R^n} \rho(x+\frac y2,x-\frac y 2)e^{-iy\cdot p/\e}dy,
$$ 
that is equal to $(2\pi\e)^n W_\e\rho$. Moreover, using \eqref{trace condition kernel}
and \eqref{eqn:margWig}, it holds
\begin{equation}
\label{eq:trace}
\textrm{tr}(\tilde\rho)=\int_{\R^{2n}} W_\e \rho(x,p)\,dx\,dp
\end{equation}

Now, we remark that the first assumption gives \eqref{eqn:disop} using Calder\'on-Vaillancourt Theorem \cite{boulkh}.

Concerning \eqref{eqn:hypH2}, we will prove that
$$
\sup_{\e \in (0,1)}\textrm{tr}(H_{\e}^2\tilde{\rho}_{0,\e}) < +\infty
$$
(as observed in Section \ref{sect:assumptions and theorem}, this condition is slightly stronger than \eqref{eqn:hypH2}).
To  this aim,
we first note that
\begin{equation}\label{carre}
H_{\e}^2=\frac{\e^4}{4}\Delta^2+U^2-\frac{\e^2}{2}\Delta U-\frac{\e^2}{2}U\Delta.
\end{equation}
Moreover, let us observe that if $\tilde\rho_1$ and $\tilde\rho_2$ have kernels $\rho_1$
and $\rho_2$ respectively, then the kernel associated to the operator $\tilde\rho_1\tilde\rho_2$
is given by $\int \rho_1(\cdot,z)\rho_2(z,\cdot)\,dz$. By this fact
and \eqref{eq:trace}, a simple computation shows that the identity
\[
\textrm{tr}(A{\rho}_{\e})=\int_{\R^{2n}} \sigma_\e(A)(x,p)W_\e\rho_{0,\e}(x,p)\,dx\,dp
\]
holds for any ``suitable'' operator $A$ (here $\sigma_\e(A)$ is the Weyl symbol of $A$).
Hence, in our case,
\[
\textrm{tr}(H_{\e}^2\tilde{\rho}_{\e})=\int_{\R^{2n}} \sigma_\e(H_{\e}^2)(x,p)W_\e\rho_{0,\e}(x,p)\,dx\,dp.
\]
We claim that the Weyl symbol of $H_{\e}^2$ is
$$
\sigma_\e(H_{\e}^2)(x,p)=\frac{|p|^4}{4}+U^2(x)+|p|^2U(x)-\frac{n\e^2}{2}\Delta U(x).
$$
Indeed, let $f(x,p):=|p|^2=\sigma_\e(-\e^2 \Delta)(x,p)$ and $g(x,p):=U(x)=\sigma_\e(U)(x,p)$.
Then, using Moyal expansion,
\begin{eqnarray*}
\sigma_\e(H_{\e}^2)(x,p) & = & \sigma_\e\left(\frac{\e^4}{4}\Delta^2+U^2-\frac{\e^2}{2}\Delta U-\frac{\e^2}{2}U\Delta\right)(x,p) \\
                         & = & \frac{f(x,p)^2}{4}+g(x,p)^2+\frac{1}{2}f\sharp g (x,p)+\frac{1}{2}g\sharp f (x,p),
\end{eqnarray*}
where by definition
$$
h_1\sharp h_2 (x,p):=\left.e^{i\frac\e 2(\partial_x\partial_{p'}-\partial_p\partial_{x'})}h_1(x,p)h_2(x',p') \right|_{x'=x,p'=p}.
$$
In our case, in the expansion of the exponential
$$
e^{i\frac\e 2(\partial_x\partial_{p'}-\partial_p\partial_{x'})}=\sum_{j \in \N} \frac{1}{j!} \left( i\frac\e 2(\partial_x\partial_{p'}-\partial_p\partial_{x'})\right)^j
$$
we can stop at the second order term, since $f(x,p)=|p|^2$. Therefore
$$
f\sharp g (x,p)=|p|^2U(x)-i\e p\cdot \nabla U(x)-\frac{n\e^2}{2}\Delta U(x),
$$
and
$$
g\sharp f (x,p)=|p|^2U(x)+i\e p\cdot \nabla U(x)-\frac{n\e^2}{2}\Delta U(x).
$$
This proves the claim and conclude the proof of the proposition.

%
%
%

\end{proof}

\section{Proof of Theorem \ref{thm:mainthm}}
The proof of the theorem is split into several steps: first we show some basic estimates on the solutions, and we prove that the family $\tilde W_\e \rho^\e_t$
is tight in space and uniformly weakly continuous in time (this is the compactness part).
Then we show that $\tilde W_\e \rho^\e_t$ solves the Liouville equation
(away from the singular set of the Coulomb potential) with an error term which converges to zero as $\e\to 0$. Combining this fact with some uniform decay estimate for $\tilde W_\e \rho^\e_t$
away from the singularity, we finally prove that any limit point is bounded and solves the Liouville equation.
By the uniqueness of solution to the Liouville equation in the function space $L^{\infty}_+([0,T];L^{1}(\R^{2n})\cap L^{\infty}(\R^{2n}))$,
we conclude the desired result.

Let us observe that some of our estimates can be found \cite{amfrja} and \cite{AFFGP}.
However, the setting and the notation there are
slightly different, and in some cases one would have to recheck the
details of the proofs in \cite{amfrja,AFFGP}
to verify that everything works also in our case. Hence,
for sake of completeness and in order to make this paper more accessible, we have decided to include all the details.

\subsection{Basic estimates}
\subsubsection{Conserved quantities}
The spectral decomposition of $\tilde{\rho}^{\e}_t$ is
$$
\tilde{\rho}^{\e}_t=\sum_{j \in\N} \mu_j^{(\e)} \langle \phi_{j,t}^{(\e)}, \cdot \rangle \phi_{j,t}^{(\e)},
$$
where $\phi_{j,t}^{(\e)}=e^{-itH_{\e}/\e}\phi_{j}^{(\e)}$ solves \eqref{eq:Sch}.
By standard results on the unitary propagator $e^{-itH_{\e}/\e}$ follows that
\begin{equation}\label{eqn:meanH}
\sum_{j \in \N} \mu_j^{(\e)} \langle \phi_{j,t}^{(\e)}, H_{\e}  \phi_{j,t}^{(\e)}\rangle=\sum_{j \in \N} \mu_j^{(\e)} \langle \phi_{j}^{(\e)}, H_{\e}  \phi_{j}^{(\e)}\rangle
\end{equation}
and
\begin{equation}\label{meanH2}
\sum_{j \in \N} \mu_j^{(\e)} \| H_{\e}\phi_{j,t}^{(\e)}\|^2=\sum_{j \in \N} \mu_j^{(\e)}  \|H_{\e} \phi_{j}^{(\e)}\|^2
\end{equation}
for all $t \in \R$ and $\e\in(0,1)$. Therefore, using (\ref{eqn:hypH2}) we have
\begin{equation}\label{eqn:supemeanH}
\sup_{\e\in(0,1)} \sup_{ t \in \R} \sum_{j \in \N} \mu_j^{(\e)} \langle \phi_{j,t}^{(\e)}, H_{\e}  \phi_{j,t}^{(\e)}\rangle < +\infty,
\end{equation}
\begin{equation}\label{eqn:supemeanH2}
\sup_{\e\in(0,1)} \sup_{ t \in \R} \sum_{j \in \N} \mu_j^{(\e)} \|H_{\e} \phi_{j,t}^{(\e)}\|^2 < +\infty.
\end{equation}

\subsubsection{A priori estimates}
From (\ref{eqn:meanH}), (\ref{meanH2}) and from the fact that $U_s >0$ and $U_b \in L^{\infty}(\R^n)$, follows that for all $\e\in(0,1)$
\begin{equation}
\sup_{t \in \R} \int_{\R^n} U_s^2(x)\rho_{t}^{\e}(x,x) \; dx \leq \sum_{j \in \N} \mu_j^{(\e)} \|H_{\e} \phi_{j}^{(\e)}\|^2 +2 \|U_b\|_{\infty}\left( \sum_{j \in \N} \mu_j^{(\e)} \langle \phi_{j}^{(\e)}, H_{\e} \phi_{j}^{(\e)}\rangle +\|U_b\|_{\infty}\right)
\end{equation}
and
\begin{equation}
\sup_{t \in \R} \frac{1}{2}\sum_{j \in \N} \mu_j^{(\e)}\int_{\R^n} |\e \nabla \phi_{j,y}^{(\e)}(x)|^2 \; dx \leq \sum_{j \in \N} \mu_j^{(\e)}
\langle \phi_j^{(\e) }H_{\e} \phi_{j}^{(\e)} \rangle + \|U_b\|_{\infty}.
\end{equation}
Hence, by (\ref{eqn:supemeanH}) and (\ref{eqn:supemeanH2}) we obtain
\begin{equation}\label{defn:C1}
\sup_{\e\in(0,1)} \sup_{t \in \R} \int_{\R^n} U_s^2(x)\rho_{t}^{\e}(x,x) \; dx \leq C_1
\end{equation}
and
\begin{equation}\label{eqn:epsgrad2}
\sup_{\e\in(0,1)} \sup_{t \in \R}\sum_{j \in \N} \mu_{j}^{(\e)}\int_{\R^n} |\e \nabla \phi_{j,t}^{(\e)}(x)|^2 \; dx \leq C_2.
\end{equation}


\subsubsection{Propagation of (\ref{eqn:disop}) and consequences}
Observe that, by unitarity of  $e^{itH_{\e}/\e}$, we have, for all $t\in \R$,
\begin{equation}\label{eqn:disopt}
\frac{1}{\e^n} \tilde{\rho}_t^{\e} \leq C \,\textrm{Id}.
\end{equation}
Hence, since
\begin{equation}\label{eqn:limHus}
\tilde{W}_{\e}\rho_{t}^{\e}(y,p)=\frac{1}{(2\pi)^n} \langle \phi^{\e}_{y,p}, \tilde{\rho}_{t}^{\e}\phi^{\e}_{y,p} \rangle,
\end{equation}
(see Appendix), using (\ref{eqn:disopt}) we have
\begin{equation}\label{eqn:esthus}
 \sup_{ \e\in(0,1)} \sup_{t\in \R}\|\tilde{W}_{\e}\rho_{t}^{\e}\|_{\infty} \leq \frac{C \e^n}{(2\pi)^n}\|\phi^{\e}_{y,p}\|^2=\frac{C}{(2\pi)^n}
\end{equation}
(because $\|\phi^{\e}_{y,p}\|=\e^{-n/2}$).
Now, define for all $x,y \in \R^n$ and $\e, \lambda > 0$
$$
g_{\e,\lambda,y}(x)=(\sqrt{2}\e)^{n/2}(\pi \lambda)^{n/4} G_{\lambda \e^2}^{(n)}(x-y).
$$
Observe that
\begin{eqnarray}
\frac{1}{\e^n}\langle g_{\e,\lambda,y}, \tilde{\rho}_t^{\e} g_{\e,\lambda,y} \rangle & = & \frac{1}{\e^n} \sum_{j \in\N} \mu_j^{(\e)}|\langle g_{\e,\lambda,y}, \phi_{j,t}^{(\e))} |^2 \nonumber\\
                                                                                     & = & \frac{1}{\e^n} \sum_{j \in\N} \mu_j^{(\e)}|(\sqrt{2}\e)^{n/2}(\pi \lambda)^{n/4}\phi_{j,t}^{(\e))} \ast G_{\lambda \e^2}^{(n)} (y)  |^2 \nonumber\\
                                                                                     & = & 2^{n/2}(\pi \lambda)^{n/2}\sum_{j \in\N} \mu_j^{(\e)}|\phi_{j,t}^{(\e))} \ast G_{\lambda \e^2}^{(n)} (y)  |^2, \nonumber
\end{eqnarray}
therefore, since $\|g_{\e,\lambda,y}\|=1$, by (\ref{eqn:disopt}) we have that
\begin{equation}
2^{n/2}(\pi \lambda)^{n/2}\sum_{j \in\N} \mu_j^{(\e)}|\phi_{j,t}^{(\e))} \ast G_{\lambda \e^2}^{(n)} (y)  |^2 \leq C.
\end{equation}
So
\begin{equation}\label{eqn:eps2}
\sup_{\e \in (0,1)} \sup_{t \in [0,T]}\sup_{y \in \R^n} \; \sum_{j \in\N} \mu_j^{(\e)}|\phi_{j,t}^{(\e)} \ast G_{\lambda \e^2}^{(n)} (y)  |^2 \leq \frac{C}{\lambda^{n/2}}.
\end{equation}

\subsection{Tightness in space}
Define $C_{R}^{(k)}=\{y=(y_1,\dots,y_k) \in \R^k : |y_j| \leq R, j=1, \dots, k\}$. We want to prove that
\begin{equation}\label{tightnesshus}
\lim_{R \to +\infty} \sup_{\e \in(0,1)} \sup_{t \in [0,T]} \int_{\R^{2n} \setminus C_R^{(2n)}} \tilde{W}_{\e}\rho_{\e}^t(x,p) \,dx \,dp =0.
\end{equation}
Observe that for all $R>0$
\begin{eqnarray*}
\sup_{\e \in(0,1)} \sup_{t \in [0,T]} \int_{\R^{2n} \setminus C_R^{(2n)}} \tilde{W}_{\e}\rho_{\e}^t(x,p) \,dx \,dp  & \leq & \sup_{\e \in(0,1)} \sup_{t \in [0,T]} \frac{1}{2}\left[ \int_{(\R^n \setminus C_R^{(n)})\times \R^n} \tilde{W}_{\e}\rho_{\e}^t(x,p) \,dx \,dp \right. \\
                                                                                                             &      & \left. + \int_{\R^n \times (\R^n \setminus C_R^{(n)})} \tilde{W}_{\e}\rho_{\e}^t(x,p) \,dx \,dp \right],
\end{eqnarray*}
so we can check the tightness property separately for the first and the second marginals of $\tilde{W}_{\e}\rho_{\e}^t$.
From (\ref{eqn:exlimmeas}) follows immediately that the family $\{\tilde{W}_{\e}\rho_{\e,0}\Leb{2n}\}_{\e \in (0,1)}$ is tight
(because, by Prokhorov's Theorem, a family of nonnegative finite measures on $\R^{2n}$ is tight if and only if it is relatively compact in the duality with $C_b(\R^{2n})$). Therefore
\begin{equation}\label{eqn:husmargI}
\lim_{R \to +\infty} \int_{(\R^{n} \setminus C_R^{(n)}) \times \R^{n}} \tilde{W}_{\e}\rho_{\e,0}(x,p) \,dx \,dp=0.
\end{equation}
Let $\chi \in C(\R^n)$, $0\leq \chi \leq 1$ such that $\chi(x)=0$ if $|x|<1/2$ and $\chi(x)=1$ if $|x|>1$, and define
$\chi_R(x):=\chi(x/R)$. Observe that
$\|\nabla \chi_R\|_{\infty} \leq C'/R$ and $\|\Delta \chi_R\|_{\infty} \leq C'/R^2$.
We define the following operator:
$$
A_R^{(\e)}\psi(x)=\chi_{R}\ast G_{\e}^{(n)}(x)\psi(x), \qquad \psi \in L^2(\R^n).
$$
Observe that
$$
\frac{d}{dt}\textrm{tr}(A_R^{(\e)}\tilde{\rho}_{\e}^t)=-\frac{i}{\e}\textrm{tr}([A_R^{(\e)},H_{\e}]\tilde{\rho}_{\e}^t)
$$
and that $[A_R^{(\e)},H_{\e}]=\e^2\bigl(\Delta(\chi_{R}\ast G_{\e}^{(n)})/2+\nabla(\chi_{R}\ast G_{\e}^{(n)})\cdot \nabla\bigr)$. So, using \eqref{eqn:epsgrad2},
\begin{eqnarray*}
\frac{d}{dt} \textrm{tr}(A_R^{(\e)}\tilde{\rho}_{\e}^t) & = & \frac{d}{dt} \int_{\R^{2n}} \chi_{R}(x)\tilde{W}_{\e}\rho_{\e}^t(x,p) \,dx \,dp \\
                                                       & \leq & \frac{C'\e}{R^2}+\frac{C'\sqrt{C_2}}{R} \leq \frac{C'}{R^2}+\frac{C'\sqrt{C_2}}{R},
\end{eqnarray*}
which gives
\begin{eqnarray*}
\int_{(\R^n \setminus C_{2R}^{(n)})\times \R^n} \tilde{W}_{\e}\rho_{\e}^t(x,p) \,dx \,dp & \leq & \int_{\R^{2n}} \chi_{R}(x)\tilde{W}_{\e}\rho_{\e}^t(x,p) \,dx \,dp \\
                                                                                        & \leq & \int_{\R^{2n}} \chi_{R}(x)\tilde{W}_{\e}\rho_{0,\e}(x,p) \,dx \,dp +\biggl[\frac{C'}{R^2}+\frac{C'\sqrt{C_2}}{R}\biggr]T \\
                                                                                        & \leq & \int_{(\R^n \setminus C_{R}^{(n)})\times \R^n} \tilde{W}_{\e}\rho_{0,\e}(x,p) \,dx \,dp + \biggl[\frac{C'}{R^2}+\frac{C'\sqrt{C_2}}{R}\biggr]T.
\end{eqnarray*}
Therefore, using (\ref{eqn:husmargI}), we get
\begin{equation}
\lim_{R \to +\infty }\sup_{\e \in (0,1)} \sup_{t \in [0,T]}\int_{(\R^n \setminus C_{2R}^{(n)})\times \R^n} \tilde{W}_{\e}\rho_{\e}^t(x,p) \,dx \,dp =0,
\end{equation}
as desired.
For the second marginal we observe first that
\begin{equation}\label{eqn:2marghus}
\int_{\R^{2n}}|p|^2 \tilde{W}_{\e}\rho_{\e}^t(x,p) \,dx \,dp = \int_{\R^{2n}}|p|^2 W_{\e}\rho_{\e}^t(x,p) \,dx \,dp +\frac{n\e}{2}
\end{equation}
and
\begin{eqnarray*}
\int_{\R^{2n}}|p|^2W_{\e}\rho_{\e}^t(x,p) \,dx \,dp & = & \sum_{j \in \N} \mu_j^{(\e)}\int_{\R^n}\left| \frac{1}{(2\pi \e)^{\e/2}}\hat{\phi}_{j,t}^{(\e)}\left(\frac{p}{\e}\right)\right|^2|p|^2 \; dp \nonumber \\
                                                   & = & \sum_{j \in \N} \mu_j^{(\e)}\int_{\R^n} |\e \nabla \phi_{j,t}^{(\e)}(x)|^2\; dx
\end{eqnarray*}
therefore, using (\ref{eqn:2marghus}) and (\ref{eqn:epsgrad2}), we have that
\begin{equation}
\label{eq:bound p^2}
\sup_{\e \in (0,1)} \sup_{t \in [0,T]}\int_{\R^{2n}}|p|^2 \tilde{W}_{\e}\rho_{\e}^t(x,p) \,dx \,dp  \leq C_2+\frac{n}{2}
\end{equation}
and so
$$
0\leq \sup_{\e \in (0,1)} \sup_{t \in [0,T]}\int_{\R^{n}\times (\R^n \setminus C^{(n)}_R)}  \tilde{W}_{\e}\rho_{\e}^t(x,p) \,dx \,dp  \leq \frac{1}{R^2}\left(C_2+\frac{n}{2}\right) \to 0 \qquad\mbox{as } R \to +\infty.
$$

\subsection{Weak Lipschitz continuity in time}
Here we prove that for all $\phi \in C^{\infty}_c(\R^{2n})$ the map
$$
t \in \R \mapsto f_{\e, \phi}(t):=\int_{\R^{2n}}\phi(x,p) \tilde{W}_{\e}\rho_t^{\e}(x,p)\,dx \,dp
$$
is differentiable and
\begin{equation}\label{eqn:equicont}
\sup_{\e \in (0,1)}\sup_{t \in \R} \left|\frac{d}{dt}f_{\e, \phi}(t)\right| \leq C_{\phi},
\end{equation}
where $C_{\phi}$ is a constant depending only on $\phi$. First observe that
\begin{equation}
f_{\e, \phi}(t)=\int_{\R^{2n}}W_{\e}\rho_t^{\e}(x,p) \phi_{\e}(x,p)\,dx \,dp,
\end{equation}
where $\phi_{\e}:=\phi \ast G_{\e}^{(2n)}$. Therefore, using (\ref{Wigner_calc_in}), we have
\begin{eqnarray}
\frac{d}{dt}f_{\e, \phi}(t) & = & \int_{\R^{2n}}\Errordelta{U_b}{\rho^\eps_t}(x,p)\phi_{\e}(x,p)\,dx \,dp \nonumber \\
                            &   & + \int_{\R^{2n}}\Errordelta{U_s}{\rho^\eps_t}(x,p)\phi_{\e}(x,p)\,dx \,dp \nonumber \\
                            &   & + \int_{\R^{2n}} (p \cdot \nabla_x \phi_{\e}(x,p) )W_{\e}\rho_t^{\e}(x,p) \,\,dx \,dp.
\end{eqnarray}
For the first term it is easy to check that
\begin{equation}\label{eqn:errordeltaUb1}
\left|\int_{\R^{2n}}\Errordelta{U_b}{\rho^\eps_t}(x,p)\phi_{\e}(x,p)\,dx \,dp \right| \leq \frac{\| \nabla U_b\|_{\infty}}{(2\pi)^n} \int_{\R^n}|y|\sup_{x \in \R^n}|{\cal F}_p\phi_{\e}|(x,y)\,dy.
\end{equation}

In the case of the Coulomb potential we follow a specific argument
borrowed from \cite[proof of Theorem~1.1(ii)]{amfrja}), based on the
inequality
\begin{equation}\label{euclidnew}
\biggl|\frac{1}{|z+w/2|}-\frac{1}{|z-w/2|}\biggr| \leq
\frac{|w|}{|z+w/2||z-w/2|}
\end{equation}
with $z=(x_i-x_j)\in\R^3$, $w=\e(y_i-y_j)\in\R^3$. By estimating the
difference quotients of $U_s$ as in \eqref{euclidnew}, using \eqref{defn:C1} we obtain
\begin{eqnarray}\label{secondaerrorn}
\biggl|\int_{\R^{2n}}\Errordelta{U_{s}}{\rho^{\e}_t}(x,p)\phi_{\e}(x,p)\,dx\,dp\biggr| &\leq & C_*\int_{\R^n}|y|\sup_{x'\in \R^n}|{\cal F}_p\phi_{\e}(x',y)|\,dy\int_{\R^n}U_{s}^2(x)\rho_{t}^{\e}(x,x)\,dx \nonumber \\
                                                                            & \leq &  C_* C_1 \int_{\R^n}|y|\sup_{x' \in \R^n}|{\cal
F}_p\phi_{\e}(x',y)|\,dy,
\end{eqnarray}
with $C_*$ depending only on the numbers $Z_1, \dots ,Z_M$, and $C_1$ is the constant defined in (\ref{defn:C1}).

For the last term it is easy to see that
\begin{equation}
\left| \int_{\R^{2n}} (p \cdot \nabla_x \phi_{\e}(x,p) )W_{\e}\rho_t^{\e}(x,p) \,\,dx \,dp \right| \leq \frac{1}{(2\pi)^n}  \int_{\R^n} \sup_{x' \in \R^n}|{\cal F}_p \tilde{\phi}_{\e}|(x',y)\,dy,
\end{equation}
where
$$
\tilde{\phi}_{\e}(x,p)=p \cdot \nabla_x \phi_{\e}(x,p).
$$
Therefore we have only to bound
$$
\int_{\R^n}|y|\sup_{x \in \R^n}|{\cal
F}_p\phi_{\e}(x,y)|\,dy \quad \textrm{and} \quad \int_{\R^n} \sup_{x' \in \R^n}|{\cal F}_p \tilde{\phi}_{\e}|(x',y)\,dy
$$
with a constant depending only on $\phi$.

For the first term
\begin{eqnarray*}
\int_{\R^n}|y|\sup_{x \in \R^n}|{\cal
F}_p\phi_{\e}(x,y)|\,dy & = & \int_{\R^n}|y|\sup_{x \in \R^n} \left| \int_{\R^n} G_{\e}^{(n)}(x-x'){\cal
F}_p\phi(x',y)
\,dx' \right| |e^{-y^2 \e /4}| \,dy \\
                         & \leq & \int_{\R^n}|y|\sup_{z \in \R^n}\left| {\cal
F}_p\phi(z,y)   \right|\,dy \leq C^{(1)}_{\phi}.
\end{eqnarray*}
For the second term
\begin{equation}
\int_{\R^n}\sup_{x' \in \R^n}|{\cal
F}_p\tilde{\phi}_{\e}(x',y)|\,dy =  \int_{\R^n}\sup_{x\in \R^n}\left| \int_{\R^{3n}}dp \,dx' dp' e^{-ip\cdot y} \phi(x',p') G_{\e}^{(n)}(p-p') \left(p \cdot \nabla_x G_{\e}^{(n)}(x-x')\right) \right|\,dy.
\end{equation}
Now observe that
\begin{eqnarray*}
& & \int_{\R^{3n}}dp \,dx' dp' e^{-ip\cdot y} \phi(x',p') G_{\e}^{(n)}(p-p') \left(p \cdot \nabla_x G_{\e}^{(n)}(x-x')\right) \\
& = & \sum_{k=1}^{n} \int_{\R^{2n}} \,dx' dp' \; \partial_{x_k}G_{\e}^{(n)}(x-x') G_{\e}^{(n)}(p')\int_{\R^n} dp \; p_k\phi(x',p-p')e^{-ip\cdot y} \\
& = & e^{-\e y^2/4} \left[ \int_{\R^{n}} \,dx' \; (\nabla_x \cdot {\cal
F}_p g(x-x',y)) G_{\e}^{(n)}(x')  \right. \\
&  & + \frac{i \e}{2}  \left. \int_{\R^{2n}} \,dx' \; (y \cdot \nabla_x{\cal
F}_p \phi(x-x',y)) G_{\e}^{(n)}(x')  \right],
\end{eqnarray*}
where $g(x,p)=p \phi(x,p)$. Now, since $\e \in (0,1)$
\begin{eqnarray*}
\int_{\R^n}\sup_{x' \in \R^n}|{\cal
F}_p\tilde{\phi}_{\e}(x',y)|\,dy  & \leq & \int_{\R^n}\sup_{x\in \R^n} \left| \int_{\R^{n}} \,dx' \; (\nabla \cdot {\cal
F}_p g(x-x',y)) G_{\e}^{(n)}(x')  \right|  \\
                                     &      & + \frac{ \e}{2}\int_{\R^n}\sup_{x\in \R^n} \left| \int_{\R^{2n}} \,dx' \; (y \cdot \nabla{\cal
F}_p \phi(x-x',y)) G_{\e}^{(n)}(x')  \right|  \\
                                     & \leq & \int_{\R^n} dy \; \sup_{z \in \R^n} |\nabla \cdot {\cal
F}_p g(z,y)| + \frac{ \e}{2} \int_{\R^n}dy \; |y| \sup_{z \in \R^n} |\nabla_z{\cal
F}_p \phi(z,y)| \\
                                     & \leq & C_{\phi}^{(2)}.
\end{eqnarray*}

Therefore
\begin{eqnarray*}
\sup_{\e \in (0,1)}\sup_{t \in \R}\left| \frac{d}{dt} f_{\e,\phi}(t) \right|  & \leq & \frac{\|\nabla U_b\|_{\infty}}{(2\pi)^n}C_{\phi}^{(1)}+C_{*}C_1C_{\phi}^{(1)}+\frac{C_{\phi}^{(2)}}{(2\pi)^n}.
\end{eqnarray*}

\subsection{Uniform decay away from the singularity}
\label{sect:sing}
The singular set of $U_s$ is given by
\begin{equation}\label{defn:S}
S=\bigcup_{1\leq i<j\leq M} S_{ij}, \quad S_{i,j}=\left\{ x=(x_1,\dots,x_M,\bar x) \in (\R^3)^M\times \R^{n-3M} \; : \; x_i=x_j\text{ for some $i \neq j$}\right\},
\end{equation}
and we have
\begin{equation}\label{eqn:maggUs}
U_s(x) \geq \frac{c}{\textrm{dist}(x,S)},
\end{equation}
where $c >0$ depending only on $Z_1, \dots, Z_M$.
We want to prove that
\begin{equation}\label{eqn:measlim}
\limsup_{\e \to 0} \int_{\R^{2n}}\left(|p|^4+ \frac{1}{\textrm{dist}(x,S)^2} \right)\tilde{W}_{\e}\rho_{\e}^t(x,p)\,dx \,dp \leq C.
\end{equation}
We start with the second term:
\begin{eqnarray}
\int_{\R^{2n}} \,dx \,dp \frac{1}{\textrm{dist}(x,S)^2} \tilde{W}_{\e}\rho_{\e}^t(x,p) & = & \int_{B^{(n)}_R \times \R^n} \,dx \,dx' \frac{\rho_{\e}^t(x',x')G_{\e}^{(n)}(x-x')}{\textrm{dist}(x,S)^2}\nonumber \\
                                                                                             & \leq & \int_{\R^{n}} \,dx' \frac{\rho_{\e}^t(x',x')}{\textrm{dist}(x',S)^2} \nonumber \\
                                                                                             & \leq & \frac{1}{c}\int_{\R^{n}} \,dx' U_s(x')^2\rho_{\e}^t(x',x') \nonumber \\
                                                                                             & \leq & \frac{C_1}{c}, \nonumber
\end{eqnarray}
where $c$ is defined in (\ref{eqn:maggUs}), $C_1$ is defined in (\ref{defn:C1}), and we used (\ref{eqn:maggUs}).

To prove the second estimate we observe that
\begin{eqnarray*}
\int_{\R^{2n}}|p|^4 \tilde{W}_{\e}\rho_{\e}^t(x,p)\,dx \,dp &\leq& \int_{\R^{2n}}|p|^4 {W}_{\e}\rho_{\e}^t(x,p)\,dx \,dp\\
& +&\frac{n\e}{2}  \int_{\R^{2n}}|p|^2 {W}_{\e}\rho_{\e}^t(x,p)\,dx \,dp + \frac{n(n+2)\e^2}{4}.
\end{eqnarray*}
Thanks to \eqref{eq:bound p^2}, it suffices to control the first integral in the right hand side:
\begin{eqnarray*}
\int_{\R^{2n}}|p|^4W_{\e}\rho_{\e}^t(x,p) \,dx \,dp & = & \sum_{j \in \N} \mu_j^{(\e)}\int_{\R^n}\left| \frac{1}{(2\pi \e)^{\e/2}}\hat{\phi}_{j,t}^{(\e)}\left(\frac{p}{\e}\right)\right|^2|p|^4 \; dp \nonumber \\
                                                   & = & \sum_{j \in \N} \mu_j^{(\e)}\int_{\R^n} |\e^2 \Delta \phi_{j,t}^{(\e)}(x)|^2\; dx \\
& \leq  & 2 \sum_{j \in \N} \mu_j^{(\e)}\int_{\R^n} \left[|H_\e \phi_{j,t}^{(\e)}(x)|^2 + U^2(x)|\phi_{j,t}^{(\e)}(x)|^2\right]\; dx,
\end{eqnarray*}
and the last term is uniformly bounded thanks to (\ref{eqn:supemeanH2}), \eqref{defn:C1}, and the boundedness of $U_b$.

\subsection{Limit continuity equation away from the singularities}\label{limconteqn}
We want to prove that
\begin{equation}\label{eqn:cont}
\lim_{\e \to 0} \int_{0}^T \left[ \varphi'(t) \int_{\R^{2n}} \phi(x,p) \tilde{W}_{\e} \rho^t_{\e}(x,p) \,dx\,dp+ \varphi(t) \int_{\R^{2n}} \bb(x,p) \cdot \nabla \phi(x,p) \tilde{W}_{\e} \rho^t_{\e}(x,p) \,dx\,dp \right] \; dt =0
\end{equation}
for all $\phi \in C_c^{\infty}(\R^{2n} \setminus (S \times \R^n))$ and $\varphi \in C_c^{\infty}(0,T)$.
Hence, recalling (\ref{pdehusimi}), we have to show that
\begin{equation}\label{eqn:U}
\lim_{\e \to 0} \sup_{t \in [0,T]} \int_{\R^{2n}} \,dx \,dp \;\Errordelta{U}{\rho^\e_t}\ast G^{(2n)}_\e(x,p) \phi(x,p)+ \int_{\R^{2n}} \,dx \,dp \; \nabla U(x) \cdot \nabla_p \phi(x,p) \tilde{W}_{\e} \rho^t_{\e}(x,p)=0,
\end{equation}
and
\begin{equation}\label{eqn:vect}
\lim_{\e \to 0} \int_{0}^T dt \; \varphi(t)  \int_{\R^{2n}} \,dx \,dp \; \sqrt{\e}\nabla_x \cdot
[W_\e\rho^\e_t \ast \bar{G}_{\e}^{(2n)}]\phi(x,p)=0,
\end{equation}
for all $\phi \in C_c^{\infty}(\R^{2n} \setminus (S \times \R^n))$ and $\varphi \in C_c^{\infty}(0,T)$.

\subsubsection{Verification of (\ref{eqn:U})}
We can consider separately the contributions of $U_b$ and $U_s$. We start with the contribution of $U_s$. We have to prove that
\begin{equation}\label{eqn:U_s}
\lim_{\e \to 0} \sup_{t \in [0,T]} \int_{\R^{2n}} \,dx \,dp \;\Errordelta{U_s}{\rho^\e_t}\ast G^{(2n)}_\e(x,p) \phi(x,p)+ \int_{\R^{2n}} \,dx \,dp \; \nabla U_s(x) \cdot \nabla_p \phi(x,p) \tilde{W}_{\e} \rho^t_{\e}(x,p)=0
\end{equation}
for all $\phi \in C_c^{\infty}(\R^{2n} \setminus (S \times \R^n))$.

We know that
\begin{equation}\label{eqn:wigvshus}
\lim_{\e \to 0} \sup_{t \in [0,T]} \left[ \int_{\R^{2n}}\varphi(x,p) W_{\e}\rho_{\e}^t(x,p) \,dx\,dp - \int_{\R^{2n}}\varphi(x,p) \tilde{W}_{\e}\rho_{\e}^t(x,p) \,dx\,dp \right]=0
\end{equation}
for all $\varphi \in C_c^{\infty}(\R^{2n})$.

First of all, we see that we can apply
\eqref{eqn:wigvshus} with $\varphi(x,p)=\nabla U_s(x) \cdot \nabla_p \phi (x,p)$ to
replace the integrals
$$
\int_{\R^{2n}} \nabla
U_s(x) \cdot \nabla_p\phi (x,p) \tilde W_\e\psi^\e\,\,dx\,dp
$$
with
$$
\int_{\R^{2n}} \nabla
U_s(x) \cdot \nabla_p\phi (x,p)
W_\e\psi^\e\,\,dx\,dp
$$
in the verification of \eqref{eqn:U_s}.
Analogously, using \eqref{defn:C1} and \eqref{secondaerrorn} we
see that we can replace
$$
\int_{\R^{2n}}\Errordelta{U_s}{\rho_{\e}^t}\ast G^{(2n)}_\e(x,p) \phi(x,p)\,\,dx\,dp
$$
with
$$
\int_{\R^{2n}}\Errordelta{U_s}{\rho_{\e}^t}(x,p)\phi(x,p)\,\,dx\,dp.
$$
Thus, we are led to show the convergence
\begin{equation}\label{erro_co3}
\lim_{\e \to
0} \sup_{t \in [0,T]}\int_{\R^{2n}}\Errordelta{U_s}{\rho_{\e}^t}\phi\,\,dx\,dp+\int_{\R^{2n}}
\nabla U_s(x) \cdot \nabla_p\phi(x,p ) W_\e\rho_{\e}^t(x,p)
\,\,dx\,dp=0
\end{equation}
for all $\phi\in C^\infty_c((\R^{n}\setminus S) \times\R^n)\bigr)$. Since
$$
\int_{\R^{2n}}\Errordelta{U_s}{\rho_{\e}^t}\phi\,\,dx\,dp= \int_{\R^{2n}}
\frac{U_s(x+\frac{\e}{2}y)-U_s(x-\frac{\e}{2}y)}{\e}
\rho_{\e}^t\left(x+\frac{\e y}{2},x-\frac{\e y}{2}\right)  {\cal
F}_p\phi(x,y)\,dxdy
$$
we can split the region of integration in two parts, where
$\sqrt{\e}|y|>1$ and where $\sqrt{\e}|y|\leq 1$. The contribution of
the first region can be estimated as in \eqref{secondaerrorn}, with
$$
C_*\int_{\{\sqrt{\e}|y|>1\}}|y|\sup_{x'}|{\cal
F}_p\phi(x',y)|\,dy\int_{\R^n}U_{s}^2(x) \rho_{\e}^t(x,x)\,dx,
$$
which is infinitesimal, using \eqref{defn:C1} again, as $\e\to
0$. Since
$$
\frac{U_s(x+\frac{\e}{2}y)-U_s(x-\frac{\e}{2}y)}{\e}\to \nabla
U_s(x) \cdot y  $$ uniformly as $\sqrt{\e}|y|\leq 1$ and $x$ belongs
to a compact subset of $\R^n\setminus S$, the contribution of the
second part is the same of
$$
\int_{\R^{2n}} (\nabla U_s(x) \cdot y )
\rho_{\e}^t\left(x+\frac{\e y}{2},x-\frac{\e y}{2}\right){\cal
F}_p\phi(x,y)\,dxdy
$$
which coincides with
$$
\int_{\R^{2n}} \nabla U_s(x) \cdot \nabla_p\phi (x,p)
W_\e\rho_{\e}^t(x,p)\,\,dx\,dp.
$$

Now we consider the contribution of $U_b$. We have to prove that
\begin{equation}\label{erro_co}
\lim_{\e \to  0} \sup_{t \in [0,T]}
\int_{\R^{2n}}\Errordelta{U_b}{\rho^t_{\e}}(x,p)\phi_{\e}(x,p)\,\,dx\,dp+\int_{\R^{2n}}
\nabla U_b(x) \cdot \nabla_p\phi (x,p)\tilde
W_\e\rho^\e_t \,\,dx\,dp =0
\end{equation}
for all $\phi\in C^\infty_c(\R^{2n})$, where $\phi_{\e}=\phi \ast G_{\e}^{(2n)}$.
The proof of (\ref{erro_co}) is divided in two parts: first we prove that
\begin{equation}\label{erro_co2}
\lim_{\e \to  0} \sup_{t \in [0,T]}
\int_{\R^{2n}}\Errordelta{U_b}{\rho^t_{\e}}(x,p)\phi(x,p)\,\,dx\,dp+\int_{\R^{2n}}
\nabla U_b(x) \cdot \nabla_p\phi (x,p)\tilde
W_\e\rho^\e_t \,\,dx\,dp =0
\end{equation}
for all $\phi\in C^\infty_c(\R^{2n})$, and then, using the following estimate
\begin{equation}\label{eqn:errordeltaUb2}
\left|\int_{\R^{2n}}\Errordelta{U_b}{\rho^\eps_t}(x,p)\varphi(x,p)\,dx \,dp \right| \leq \frac{\| \nabla U_b\|_{\infty}}{(2\pi)^n} \int_{\R^n}|y|\sup_{x \in \R^n}|{\cal F}_p\varphi|(x,y)\,dy.
\end{equation}
for all $\varphi \in C^{\infty}_c(\R^{2n})$, we can replace $\phi$ by $\phi_{\e}$ in the first summand of (\ref{erro_co2}), obtaining (\ref{erro_co}).
The proof of (\ref{erro_co2}) is achieved by a density argument. The first
remark is that linear combinations of tensor functions
$\phi(x,p)=\phi_1(x)\phi_2(p)$, with $\phi_i\in C^\infty_c(\R^n)$,
are dense for the norm considered in \eqref{eqn:errordeltaUb2}. In this
way, we are led to prove convergence in the case when
$\phi(x,p)=\phi_1(x)\phi_2(p)$. The second remark is that
convergence surely holds if $U_b$ is of class $C^2$ (by the arguments
in \cite{lionspaul}, \cite{amfrja}). Hence,
combining the two remarks and using the linearity of the error term
with respect to the potential, we can prove convergence by a
density argument, by approximating $U_b$ uniformly and in $W^{1,2}$
topology on the support of $\phi_1$ by potentials $V_k\in C^2(\R^n)$
with uniformly Lipschitz constants; then, setting $A_k=(U_b-V_k)\phi_1$ and choosing a sequence
$\lambda_k$ in Lemma~\ref{apriorinice} converging slowly to $0$ for $k\to +\infty$, in such a way that
$\|\nabla A_k\|_2 = o(\lambda_k^{n/4})$ for $k\to +\infty$.
In this way we obtain
$$
\lim_{k\to\infty} \sup_{\e \in (0,1)} \sup_{t\in [0,T]}\int_{\R^{2n}}\Errordelta{U_b-V_k}{\rho^\e_t}(x,p)\phi_1(x)\phi_2(p)\,\,dx\,dp
=0.
$$
As for the term in \eqref{erro_co} involving the Husimi transforms,
we can use \eqref{eqn:esthus} to obtain that
\begin{eqnarray*}
& & \limsup_{k\to\infty}\sup_{\e \in (0,1)} \sup_{t\in [0,T]} \left|\int_{\R^{2n}}
\tilde W_\e\rho^\e_t \nabla
(U_b(x)-V_k(x))\cdot \nabla\phi_2(p)\phi_1(x)\,\,dx\,dp \right| \\
& & \leq \frac{C}{(2\pi)^n} \limsup_{k\to\infty}\int_{\R^n}|\phi_1(x)||\nabla U_b(x)-\nabla V_k(x)|\,dx
\int_{\R^n}|\nabla\phi_2(p)|\,dp=0.
\end{eqnarray*}
So we need only to prove the following lemma:

\begin{lemma}[A priori estimate]\label{apriorinice}
For all $\lambda>0$, we have
that
\begin{eqnarray}\label{apriori}
& & \sup_{\e \in (0,1)} \sup_{t\in [0,T]}\biggl|\int_{\R^{2n}}
\Errordelta{U_b}{\rho_t^\e}(x,p)\phi_1(x)\phi_2(p)\,\,dx\,dp
\biggr| \\
& \leq &  \|\phi_1\|_1\|\nabla
U_b\|_\infty\sup_{y \in \R^n}|y||\hat{\phi}_2(y)-\hat{\phi}_2\ast
G^{(n)}_\lambda(y)|+\sqrt{\lambda}\|\nabla
A\|_\infty\|\hat\phi_2\|_1\int_{\R^n}|u|G^{(n)}_1(u)\,du \\
& & +\frac{\sqrt{C}\|\nabla
A\|_2}{(2\pi\lambda)^{n/4}}\int_{\R^n}|z||\hat{\phi}_2|(z)\,dz+
\|U_b\|_\infty\|\nabla\phi_1\|_\infty\int_{\R^n}|y||\hat{\phi}_2\ast
G^{(n)}_\lambda|(y)\,dy
\end{eqnarray}
where $A:=U_b\phi_1$ and $C$ is the constant in \eqref{eqn:disop}.
\end{lemma}
\begin{proof} Set $\hat{\phi}_2={\cal F}_p\phi_2$. Observe that since \eqref{eqn:errordeltaUb2}
gives that
\begin{eqnarray*}
& & \sup_{\e \in (0,1)} \sup_{t\in [0,T]}\biggl|\int_{\R^{2n}}\Errordelta{U_b}{\rho^\e_t}\phi_1(x)\phi_2(p)\,\,dx\,dp-
\int_{\R^{2n}}\Errordelta{U_b}{\rho^\e_t}\phi_1(x)\phi_2(p)e^{-|p|^2\lambda}\,\,dx\,dp\biggr| \\
& \leq & \|\phi_1\|_1\|\nabla
U_b\|_\infty \sup_{y \in \R^n}|y||\hat{\phi}_2(y)-\hat{\phi}_2\ast
G^{(n)}_\lambda(y)|
\end{eqnarray*}
we recognize the first error term in
\eqref{apriori}. So we have only to estimate
\begin{equation}
\sup_{\e \in (0,1)} \sup_{t\in [0,T]}\biggl|\int_{\R^{2n}}\Errordelta{U_b}{\rho^\e_t}\phi_1(x)\phi_2(p)e^{-|p|^2\lambda}\,\,dx\,dp\biggr|.
\end{equation}
Observe that
$$
\int_{\R^{2n}}\Errordelta{U_b}{\rho^\e_t}\phi_1(x)\phi_2(p)e^{-|p|^2\lambda}\,\,dx\,dp=I_{\e,t}+II_{\e,t}-III_{\e,t},
$$
where
\begin{equation}\label{defI}
I_{\e,t}:= \int_{\R^{2n}}\frac{A(x+\frac{\e}{2} y) -A(x-\frac{\e}{2}
y)}{\e}\hat{\phi}_2\ast G^{(n)}_\lambda(y) \rho^\e_t(x+\frac{\e}{2}
y,x-\frac{\e}{2} y)\,dxdyd,
\end{equation}
\begin{equation}\label{defII}
II_{\e,t}:=\int_{\R^{2n}}U_b(x+\frac{\e}{2}y)\frac{\phi_1(x)
-\phi_1(x+\frac{\e}{2} y)}{\e}\hat{\phi}_2\ast G^{(n)}_\lambda(y)\rho^\e_t(x+\frac{\e}{2}
y,x-\frac{\e}{2} y)
\,dxdy,
\end{equation}
\begin{equation}\label{defIII}
III_{\e,t}:=-\int_{\R^{2n}}U_b(x-\frac{\e}{2}y)\frac{\phi_1(x)
-\phi_1(x-\frac{\e}{2} y)}{\e}\hat{\phi}_2\ast G^{(n)}_\lambda(y)
\rho^\e_t(x+\frac{\e}{2}
y,x-\frac{\e}{2} y)\,dxdy.
\end{equation}
Observe first that
$$
\sup_{\e \in (0,1)} \sup_{t\in [0,T]} |II_{\e,t}|+|III_{\e,t}| \leq \|U_b\|_{\infty}\|\nabla \phi_1\|_{\infty}\int_{\R^n}|y||\hat{\phi}_2\ast
G^{(n)}_\lambda|(y)\,dy.
$$
The estimate of $I_{\e,t}$ is more delicate: we first perform some manipulations of
this expression,then we estimate the resulting terms with the help of
\eqref{eqn:eps2}.

We expand the convolution product and make the change of variables
$$
u=x+\frac{\e}{2}y \qquad v=x-\frac{\e}{2} y
$$
to get
\begin{eqnarray}\label{eqn:explI}
I_{\e,t} &= & \frac{1}{(\pi\lambda)^{n/2}\e^n}\int_{\R^{3n}}\, du dv dz
\frac{A(u)-A(v)}{\epsilon} e^{-\scriptstyle{\frac{|\e
z-(u-v)|^2}{\e^2\lambda}}}
\rho^\e_t(u,v)\hat{\phi}_2(z)\nonumber \\
        & = & \frac{1}{\e} \sum_{j \in \N} \mu_j^{(\e)} \int_{\R^{2n}}(A\phi_{j,t}^{(\e)})\ast G_{\lambda \e^2}^{(n)}(v+\e z)\overline{\phi_{j,t}^{(\e)}(v)}\hat{\phi}_2(z) \,dv dz \nonumber \\
        &   & -\frac{1}{\e} \sum_{j \in \N} \mu_j^{(\e)} \int_{\R^{2n}}A(v)(\phi_{j,t}^{(\e)} \ast G_{\lambda \e^2}^{(n)})(v+\e z)\overline{\phi_{j,t}^{(\e)}(v)}\hat{\phi}_2(z) \,dv dz \nonumber \\
        & = & \frac{1}{\e} \sum_{j \in \N} \mu_j^{(\e)} \int_{\R^{2n}} \left[(A\phi_{j,t}^{(\e)})\ast G_{\lambda \e^2}^{(n)}(v+\e z) -A(v+\e z)(\phi_{j,t}^{(\e)} \ast G_{\lambda \e^2}^{(n)})(v+\e z) \right]\overline{\phi_{j,t}^{(\e)}(v)}\hat{\phi}_2(z)\,dv dz \nonumber \\
        &   & +\frac{1}{\e} \sum_{j \in \N} \mu_j^{(\e)} \int_{\R^{2n}} \left[A(v+\e z)-A(v)\right](\phi_{j,t}^{(\e)} \ast G_{\lambda \e^2}^{(n)})(v+\e z)\overline{\phi_{j,t}^{(\e)}(v)}\hat{\phi}_2(z)\,dv dz.
\end{eqnarray}
Now let us estimate the first summand in (\ref{eqn:explI})
\begin{eqnarray*}
& & \left|\frac{1}{\e} \sum_{j \in \N} \mu_j^{(\e)} \int_{\R^{2n}} \left[(A\phi_{j,t}^{(\e)})\ast G_{\lambda \e^2}^{(n)}(v+\e z) -A(v+\e z)(\phi_{j,t}^{(\e)} \ast G_{\lambda \e^2}^{(n)})(v+\e z) \right]\overline{\phi_{j,t}^{(\e)}(v)}\hat{\phi}_2(z)\,dv dz \right| \\
& = & \left| \int_{\R^{n}}dz \, \hat{\phi}_2(z)\int_{\R^{2n}}du dv \, \frac{A(v+\e z-u)-A(v+\e z)}{\e}G_{\lambda \e^2}^{(n)}(u)\overline{\phi_{j,t}^{(\e)}(v)} \phi_{j,t}^{(\e)}(v+\e z-u)\right| \\
& \leq & \|\nabla A\|_{\infty}\int_{\R^{n}}dz \, |\hat{\phi}_2(z)| \int_{\R^{2n}}du dv \,\frac{|u|}{\e}G_{\lambda \e^2}^{(n)}(u) |\overline{\phi_{j,t}^{(\e)}(v)}||\phi_{j,t}^{(\e)}(v+\e z-u)| \\
& \leq & \sqrt{\lambda}\|\nabla A\|_{\infty}  \|\hat{\phi}_2\|_1 \int_{\R^{2n}} |u|G_{1}^{(n)}(u) \, du.
\end{eqnarray*}
For the second summand in (\ref{eqn:explI}), using (\ref{eqn:eps2}), we have
\begin{eqnarray*}
& & \left| \frac{1}{\e} \sum_{j \in \N} \mu_j^{(\e)} \int_{\R^{2n}} \left[A(v+\e z)-A(v)\right](\phi_{j,t}^{(\e)} \ast G_{\lambda \e^2}^{(n)})(v+\e z)\overline{\phi_{j,t}^{(\e)}(v)}\hat{\phi}_2(z)\,dv dz \right|\\
& \leq & \sum_{j \in \N} \mu_j^{(\e)}\int_{\R^{2n}}\left| \frac{A(v+\e z)-A(v)}{\e} \right| |(\phi_{j,t}^{(\e)} \ast G_{\lambda \e^2}^{(n)})(v+\e z)||\overline{\phi_{j,t}^{(\e)}(v)}||\hat{\phi}_2(z) |\,dv dz \\
& \leq & \sqrt{\frac{C}{(2\pi \lambda)^{n/2}} } \|\nabla A\|_{2} \int_{\R^{n}} |z||\hat{\phi}_2(z) |\,dz.
\end{eqnarray*}
This completes the estimate of the term in \eqref{defI} and the
proof.
\end{proof}

\subsubsection{Verification of (\ref{eqn:vect})}
This is easy, taking
into account the fact that
$$
\int_{\R^{2n}}
W_\e\rho_\e^{t}\ast \bar{G}_\e^{(2n)}(x,p)\cdot \nabla_x\phi(x,p) \,\,dx\,dp=
\int_{\R^{2n}} W_\e\rho_\e^{t} \nabla_x\cdot
[\phi\ast\bar{G}_\e^{(2n)}]\,\,dx\,dp
$$
are uniformly bounded (recall that $\bar{G}_\e^{(2n)}$, defined in \eqref{pdehusimibis}, are probability densities).

\subsection{Proof of Theorem \ref{thm:mainthm}}

Define $\WW^{(\e)}:[0,T] \to \Probabilities{\R^{2n}}$ as $\WW^{(\e)}_t:=\tilde{W}_{\e}\rho_{\e}^t \Leb{2n}$ for all $\e \in (0,1)$ and
$t \in [0,T]$. Using (\ref{tightnesshus}), (\ref{eqn:equicont}) and Ascoli-Arzel\`a Theorem, one can prove easily that there exist a subsequence $\left\{\WW^{(\e_k)} \right\}_{k\in \N}$ and $W: [0,T] \to \Probabilities{\R^{2n}}$ such that
\begin{equation}\label{eqn:extract}
\lim_{k \to \infty}\sup_{t \in [0,T]}d_{{\mathscr
P}}(\WW^{(\e_k)}_t,W_t)=0.
\end{equation}
We now prove the following assertions:
\begin{enumerate}
\item $W: [0,T] \to \Probabilities{\R^{2n}}$ is weakly continuous and, for all $t\in [0,T]$,
$W_t=\tilde{\WW}_t \Leb{2n}$ for some function $\tilde{\WW}_t\in L^{1}(\R^{2n})\cap L^{\infty}(\R^{2n})$. Moreover $\tilde{\WW}_t \geq 0$
 and $\sup_{t \in [0,T]}\|\tilde{\WW}_t\|_{L^1(\R^{2n})}+\|\tilde{\WW}_t\|_{L^\infty(\R^{2n})}\leq C$.
In particular, $\tilde{\WW}\in L^{\infty}_+([0,T];L^{1}(\R^{2n})\cap L^{\infty}(\R^{2n}))$.

\item $\bb\in L^1_{\rm loc}\bigl((0,T)\times\R^{2n};dt\,dW_t\bigr)$, so the continuity equation (\ref{eqn:continuity}) with $\omega_t=\tilde\WW_t$ makes sense;

\item $W$ solves (\ref{eqn:continuity}) in the sense of distributions on $[0,T]\times \R^{2n}$;
\item For any $\phi \in C^\infty_c(\R^{2n})$, $t \mapsto \int_{\R^{2n}}\phi\,d W_t$ belongs to $C^1([0,T])$.
\end{enumerate}
\textit{Proof of (i):} Observe that (\ref{eqn:extract}) implies that $W: [0,T] \to \Probabilities{\R^{2n}}$ is weakly continuous because
it is uniform limit of the weakly continuous maps $\WW^{(\e_k)}$. The second part of the proposition follows immediately from (\ref{eqn:esthus}). Indeed, for all $\phi \in L^1(\R^{2n})$,
\begin{equation}\label{regconde}
\sup_{\e \in (0,1)} \sup_{t \in [0,T]}\int_{\R^{2n}} \phi(x,p)\tilde{W}_{\e}\rho^{\e}_t(x,p) \, \,dx\,dp \leq \frac{C}{(2\pi)^n}\int_{\R^{2n}}\phi(x,p)\, \,dx\,dp
\end{equation}
and so
\begin{equation}\label{regcond}
\sup_{t \in [0,T] }\int_{\R^{2n}} \phi(x,p) dW_t(x,p) \leq \frac{C}{(2\pi)^n}\int_{\R^{2n}}\phi(x,p)\, \,dx\,dp.
\end{equation}

\textit{Proof of (ii):} The estimate $\bb\in L^1_{\rm loc}\bigl((0,T)\times\R^{2n};dW_tdt\bigr)$ follows easily from (\ref{eqn:measlim}) and \eqref{eq:bound p^2}.

\textit{Proof of (iii):} First we prove that $\tilde{\WW}$ solves (\ref{eqn:continuity}) in $\R^{2n}\setminus (S \times \R^n)$, where $S$ is the singular set of $U_s$ defined in (\ref{defn:S}).
Unfortunately this does not follow  immediately by (\ref{eqn:cont}) because we have no information about the singular set $\Sigma$ of
$\nabla U_b$, so we cannot control the limit $k \to \infty$ of
$$
\int_{0}^T dt\, \varphi(t) \int_{\R^{2n}}\,dx \,dp \,\nabla U_b(x) \cdot \nabla_p\phi(x,p) \tilde{W}_{\e_k}\rho^{t}_{\e_k}(x,p)
$$
in (\ref{eqn:continuity}) with (\ref{eqn:extract}).
But we can proceed by a density argument because, using the regularity conditions (\ref{regconde}) and (\ref{regcond}), we can approximate $\nabla U_b$ in $L^1$  on supp$\phi$ by bounded continuous functions.

In order to prove that $\tilde{\WW}$ solves (\ref{eqn:continuity}) in $[0,T]\times \R^{2n}$ we use (\ref{eqn:measlim}) to obtain that
\begin{equation}\label{eqn:Wlimdist}
\sup_{t \in [0,T]} \int_{\R^{2n}}\frac{1}{\textrm{dist}(x,S)^2} \, dW_t(x,p)\,dt < +\infty .
\end{equation}
Observe that (\ref{eqn:Wlimdist}) implies that $W_t(S\times\R^n)=0$ for
every $t\in (0,T)$. The proof of the global validity of the
continuity equation uses the classical argument of removing the
singularity by multiplying any test function $\phi\in
C^\infty_c(\R^{2n})$ by $\chi_k$, where $\chi_k(x)=\chi(k{\rm
dist}(x,S))$ and $\chi$ is a smooth cut-off function equal to $0$ on
$[0,1]$ and equal to $1$ on $[2,+\infty)$, with $0\leq \chi' \leq 2$. If we use $\phi\chi_k$ as
a test function, since $\chi_k$ depends on $x$ only, we can use the
particular structure of $\bb$, namely $\bb(x,p)=(p,-\nabla U(x))$, to write the term depending on the
derivatives of $\chi_k$ as
$$
k\int_{\R^{2n}}\phi\chi'(k{\rm dist}(x,S)) p \cdot \nabla{\rm
dist}(x,S) \,dW_t(x,p) dt.
$$
If $K$ is the support of $\phi$, the integral above can be bounded by
$$
2\sup_K|p\phi|\int_{\{x\in K: k{\rm dist}(x,S)\leq
2\}}k\,dW_t(x,p)\,dt\leq\frac{8\max_K|p\phi|}{k}
\int_K\frac{1}{{\rm dist}^2(x,S)}\,dW_t(x,p),
$$
and the right hand side is infinitesimal (uniformly in $t$) as $k\to\infty$.

\textit{Proof of (iv):} Since the distributional derivative of $t \mapsto \int_{\R^{2n}}\phi W_t \,dx\,dp$ is given by $\int_{\R^{2n}}\bb\cdot \nabla \phi \,dW_t$,
we have to show that the map
$$
t \mapsto \int_{\R^{2n}}\bb\cdot \nabla \phi\,d W_t
$$
is continuous. Observing that the map $t \mapsto W_t$ is weakly continuous and $W_t =\tilde{\WW}_t \Leb{2n}$ with $\tilde{\WW}\in L^{\infty}_+([0,T];L^{1}(\R^{2n})\cap L^{\infty}(\R^{2n}))$,
the only delicate term is
$$
\int_{\R^{2n}}\nabla U_s(x)\cdot \nabla_p \phi(x,p)\,d W_t.
$$
Define the nonnegative Hamiltonian function $\mathcal{H}=|p|^2/2 + U + \|U_b\|_\infty$. Taking the limit in \eqref{eqn:measlim} as $\e\to 0$ we easily deduce that
$$
\sup_{t \in [0,T]}\int_{\R^{2n}} \mathcal{H}^2\,d W_t \leq C \sup_{t \in [0,T]}\int_{\R^{2n}} \Bigl(1+|p|^4+U_s^2(x)\Bigr)\,d W_t <+\infty.
$$
Since the Hamiltonian is preserved by the Liouville dynamics (under our assumptions on the potential, this fact is contained in the proof of \cite[Theorem 6.1]{AFFGP}), the above bound implies
$$
\sup_{t \in [0,T]}\int_{\{\mathcal{H}\geq N\}} \mathcal{H}^2 \,dW_t =\int_{\{\mathcal{H}\geq N\}} \mathcal{H}^2 \,dW_0 \to 0 \qquad \text{as $N \to \infty$.}
$$
As $U_s \leq \mathcal{H}$, this implies
$$
\sup_{t \in [0,T]}\int_{\{U_s \geq N\}} U_s^2\, dW_t \leq \int_{\{\mathcal{H}\geq N\}} H^2 \,dW_0 \to 0 \qquad \text{as $N \to \infty$.}
$$
Hence, if we define the sets $A_N:=\{U_s \leq N\}$, the functions
$$
t \mapsto f_N(t):=\int_{A_N}\nabla U_s(x)\cdot \nabla_p \phi \,dW_t
$$
are continuous and converge uniformly to $\int_{\R^{2n}}\nabla U_s(x)\cdot \nabla_p \phi \,dW_t$  as $N \to \infty$. This proves (iv).

To conclude the proof of the theorem, recalling that $\WW$ denote the unique distributional solution of (\ref{eqn:continuity})
in $L^{\infty}_+([0,T];L^{1}(\R^{2n})\cap L^{\infty}(\R^{2n}))$ starting from $\bar \omega\Leb{2n}$ (see \cite[Theorem 6.1]{AFFGP}), we have proved $\tilde{\WW}=\WW$, and so
\begin{equation}
\lim_{k \to \infty}\sup_{t \in [0,T]}d_{{\mathscr
P}}(\tilde{W}_{\e_k}\rho^t_{\e_k}\Leb{2n},\WW_t\Leb{2n})=0.
\end{equation}
Since the limit $\WW_t\Leb{2n}$ is independent of the chosen subsequence, this implies
the convergence of the whole family, namely
\begin{equation}\label{eqn:family}
\lim_{\e \to 0}\sup_{t \in [0,T]}d_{{\mathscr
P}}(\tilde{W}_{\e}\rho^t_{\e}\Leb{2n},\WW_t\Leb{2n})=0,
\end{equation}
as desired.

\appendix
\section{Notations and some notions about density operators}
A density operator on $L^2(\R^n)$ is a positive, self-adjoint, trace-class operator, namely $\tilde{\rho}=\tilde{\rho}^*, \tilde{\rho}\geq 0 $ and $\textrm{tr}(\tilde{\rho})=1$, where the trace is defined as follows:
\begin{equation}\label{defn:trace}
\textrm{tr}(\tilde{\rho}):=\sum_{j\in \N} \langle \varphi_j, \tilde{\rho}\varphi_j \rangle
\end{equation}
with $\{\varphi_j\}_{j \in \N}$ is any orthonormal basis of $L^2(\R^n)$. It can be shown that each density operator $\tilde{\rho}$ is a compact operator, so it can be decomposed as follows
\begin{equation}
\label{eq:spectral dec}
\tilde{\rho}=\sum_{j\in \N} \lambda_j \langle \psi_{j}, \cdot \rangle \psi_{j}
\end{equation}
where $0\leq \lambda_1 \leq \lambda_2 \leq \dots \leq 1$, and $\{\psi_{j}\}_{j\in \N}$ is a orthonormal basis of eigenvectors of $\tilde{\rho}$. Therefore $\tilde{\rho}$ is an integral operator
and its kernel is
$$
\rho(x,y)=\sum_{j\in \N} \lambda_j  \psi_{j}(x)\overline{ \psi_{j}(y)},
$$
so that
$$
\tilde{\rho}\psi(x)=\int_{\R^n} \rho(x,y) \psi(y) \;dy.
$$
Observe that the trace condition on $\tilde{\rho}$ can be expressed as follows in terms of its kernel
\begin{equation}\label{trace condition kernel}
\textrm{tr}(\tilde{\rho})=\int_{\R^n}\rho(x,x) \, dx=1.
\end{equation}
The Wigner transform of $\rho$ is defined as
$$
W_\e\rho(x,p):=\frac{1}{(2\pi)^n}\int_{\R^n}\rho(x+\frac{\e}{2}y,x-\frac{\e}{2}y)e^{-ipy}dy,
$$
and the Husimi transform of $\rho$ as
$$
\tilde{W}_\e\rho:=W_\e\rho \ast G_{\e}^{(2n)}, \qquad
G_{\e}^{(2n)}(x,p):=G_{\e}^{(n)}(x)G_{\e}^{(n)}(p)=\frac{e^{-\frac{(|x|^2+|p|^2)}{\e}}}{(\pi \e)^n}.
$$
It is easy to check that the marginals of $W_\e\rho$ are
\begin{equation}\label{eqn:margWig}
\int_{\R^n} W_\e\rho(x,p) \; dp = \rho(x,x) \qquad \textrm{and} \qquad \int_{\R^n} W_\e\rho(x,p) \; dx = \frac{1}{(2\pi \e)^n}\mathcal{F}\left(\frac{p}{\e},\frac{p}{\e}\right)
\end{equation}
where
\begin{equation}
\label{eq:Fourier}
\mathcal{F}\rho\left(q,q \right)=\int_{\R^n} \rho(u,u) e^{-iq\cdot u } \; du.
\end{equation}
Similarly the marginals of $\tilde{W}_\e\rho$ are
\begin{equation}\label{eqn:margHusx}
\int_{\R^n} \tilde{W}_\e\rho(x,p) \; dp = \int_{\R^n} \rho(x-x',x-x')G_{\e}^{(n)}(x') \,dx'
\end{equation}
and
\begin{equation}\label{eqn:margHusp}
\int_{\R^n} \tilde{W}_\e\rho(x,p) \; dx = \frac{1}{(2\pi \e)^n}\int_{\R^n}\mathcal{F}\rho\left(\frac{p-p'}{\e},\frac{p-p'}{\e}\right)G_{\e}^{(n)}(p') \; dp'.
\end{equation}
Moreover, the Husimi transform is
nonnegative: indeed (see for instance \cite{lionspaul}),
\begin{equation}
\label{eq:husimi}
\tilde
W_\e\psi(x,p)=\frac{1}{\e^n}|\<\psi,\phi^\e_{x,p}\>|^2,
\end{equation}
where $\<\cdot,\cdot\>$ is the scalar product on $L^2(\R^n)$ and
$$
\phi^\e_{x,p}(y):=\frac{1}{(\pi\e)^{n/4}}
e^{-|x-y|^2/(2\e)}e^{-i(p\cdot y)/\e} \in L^2(\R^n), \qquad
\|\phi^\e_{x,p}\|=1.
$$
Hence $\tilde W_\e\psi \geq 0$, and using the spectral decomposition \eqref{eq:spectral dec}
one obtains the non-negativity of $\tilde W_\e\rho$ for any trace-class operator $\rho$.
Moreover, combining \eqref{trace condition kernel} and \eqref{eqn:margHusx}, it follows that $\tilde W_\e\rho$ is a probability measure.

\signaf
\signml
\signtp

\end{document}